\documentclass[pdflatex,sn-mathphys-num]{sn-jnl}

\usepackage{graphicx}%
\usepackage{multirow}%
\usepackage{amsmath, amsthm, amssymb,amsfonts}%
\usepackage{cleveref}
\usepackage{mathrsfs}%
\usepackage[title]{appendix}%
\usepackage{xcolor}%
\usepackage{textcomp}%
\usepackage{manyfoot}%
\usepackage{booktabs}%
\usepackage{algorithm}%
\usepackage{algorithmicx}%
\usepackage{algpseudocode}%
\usepackage{listings}%
\usepackage{tikz}
\usepackage{bookmark}
\usetikzlibrary{decorations.pathreplacing}
\usepackage{subcaption}
\usepackage[percent]{overpic}
\usepackage{comment}
\usepackage{mathtools}
\DeclarePairedDelimiter{\floor}{\lfloor}{\rfloor}
\usepackage{array,arydshln,makecell,mathtools}
\usepackage{enumitem}

\definecolor{fista}{RGB}{77, 190, 238}
\definecolor{itta}{RGB}{237, 177, 32}
\definecolor{mfista}{RGB}{0, 114, 189}
\definecolor{mitta}{RGB}{217, 83, 25}
\definecolor{npd}{RGB}{54, 200, 91}
\definecolor{mnpd}{RGB}{2, 140, 45}
\definecolor{pnpd}{RGB}{245, 105, 180}
\definecolor{mpnpd}{RGB}{228, 0, 124}

\theoremstyle{thmstyleone}%
\newtheorem{theorem}{Theorem}
\newtheorem{example}{Example}%
\newtheorem{remark}{Remark}%
\newtheorem{lemma}[theorem]{Lemma}
\newtheorem{proposition}[theorem]{Proposition}%
\newtheorem{definition}{Definition}%
\newtheorem{assumption}{Assumption}

\newcommand{\R}{\mathbb{R}}

\DeclareMathOperator*{\argmin}{argmin}
\DeclareMathOperator*{\argmax}{argmax}
\DeclareMathOperator{\relint}{relint}

\newcommand{\prox}{\operatorname{prox}}

\newcommand{\dom}{\operatorname{dom}}

\begin{document}

\title[Article Title]{Multilevel Preconditioning Strategies for Convex Optimization Methods in Image Deblurring}

\author[1]{\fnm{Stefano} \sur{Aleotti}}\email{stefano.aleotti@uninsubria.it}

\author[1]{\fnm{Claudia} \sur{Binda}}\email{cbinda2@uninsubria.it}

\author[1]{\fnm{Marco} \sur{Donatelli}}\email{marco.donatelli@uninsubria.it}

\author[2,3]{\fnm{Rolf} \sur{Krause}}\email{rolf.krause@kaust.edu.sa}

\affil[1]{\orgdiv{Department of Science and High Technology}, \orgname{University of Insubria}, \orgaddress{\street{Via Valleggio 11}, \city{Como}, \postcode{22100}, \country{Italy}}}

\affil[2]{\orgdiv{AMCS, CEMSE}, \orgname{King Abdullah University of Science and Technology}, \orgaddress{\city{Thuwal}, \postcode{23955-6900}, \country{Kingdom of Saudi Arabia}}}

\affil[3]{\orgdiv{Euler Institute and Faculty of Informatics}, \orgname{Università della Svizzera italiana}, \street{Via Giuseppe Buffi 10}, \city{Lugano}, \postcode{6900}, \country{Switzerland}}

\abstract{Proximal gradient methods are widely used in imaging, and their speed of convergence can be accelerated by incorporating variable metrics and/or extrapolation steps. Recent works have shown that preconditioning strategies can significantly enhance this acceleration, in particular, for image deblurring problems. In parallel, a multilevel framework has been introduced to speed up inertial and inexact forward–backward schemes for image restoration problems.\\
In this paper, we combine preconditioning and multilevel strategies to design a robust and consistent acceleration framework for both standard and inexact forward-backward schemes applied to regularized convex optimization problems. Numerical experiments in image deblurring confirm that our approach yields a substantial improvement
in convergence speed compared to standard methods.}


\maketitle

\section{Introduction}\label{sec1}
In this manuscript, we focus on image deblurring problems described by the model equation
\begin{equation}\label{eq:Model_Equation}
    Au = b^\delta,
\end{equation}
where \( A\in\mathbb{R}^{s\times d} \) represents the discretization of a space-invariant convolution operator, \( b^{\delta}\in\mathbb{R}^{s} \) denotes the observed image corrupted by white Gaussian noise $\eta_{\delta}$, and \( u\in\mathbb{R}^d \) is the unknown two-dimensional image with \( d \) pixels. We assume that \( b^{\delta} \) satisfies
\begin{equation}\label{eq:bdelta}
    b^{\delta} = b + \eta_{\delta}, \qquad \|b^{\delta} - b\| \leq \delta,
\end{equation}
where $b$ denotes the unobserved noise-free data, $\|\cdot\|$ is the Euclidean norm, and $\delta > 0$ is an upper bound on the noise level.

The ill-posedness of the operator \( A \), together with the presence of noise, requires the use of regularization in order to solve \eqref{eq:Model_Equation}; see \cite{HNO,Engl1996-yp,Bertero1998b}. Regularization based on a standard variational approach consists of solving the optimization problem
\begin{equation}\label{eq:problem_leastsquares}
     \underset{u\in\mathbb{R}^d}{\operatorname{argmin}} \ f(u) + g(Wu),
\end{equation}
where \( f:\mathbb{R}^d\rightarrow \mathbb{R} \) is convex and smooth, \( g:\mathbb{R}^{d'}\rightarrow \mathbb{R}\cup\{\infty\} \) is convex and possibly non-smooth, and \( W\in\mathbb{R}^{d'\times d} \) is a linear operator. 
Here, \( f \) denotes the \textit{data-fidelity} term, depending on $A$ and measuring the discrepancy between the observed data and the model, while the possibly non-differentiable term \( g \circ W = g(W \cdot)\) serves as a penalty that incorporates prior information about the reconstructed image. 
Traditionally, this penalty term is weighted by an explicit regularization parameter $\lambda>0$ to balance the two terms. 
In this paper, in agreement with the white Gaussian noise assumption on the observed data, we consider the particular instance of the model problem~\eqref{eq:problem_leastsquares} where
\begin{equation}\label{eq:fLS}
   f(u) = \frac{1}{2}\|Au - b^{\delta}\|^2. 
\end{equation}

A particularly effective way to compute approximate solutions of \eqref{eq:problem_leastsquares} is to employ proximal gradient methods \cite{FISTA,Combettes-Wajs-2005,Daubechies-et-al-2004}, which are first-order iterative algorithms. These methods typically achieve mild-to-moderate accuracy at a relatively low cost per iteration, especially when the dimension \( d \) is large. Their basic iteration consists of alternating a gradient step on the differentiable part \( f \) with a proximal step on the non-smooth term \( g\circ W \), namely
\begin{align}
    u^{k+1} & = \operatorname{prox}_{\alpha g\circ W}(u^k - \alpha \nabla f(u^k)),
    \label{eq:pg_method}
\end{align}
where $\alpha > 0$ is the step-length parameter along the descent direction \( -\nabla f(u^k) \), and \( \operatorname{prox}_{\alpha g\circ W} \) denotes the proximal operator associated with the non-smooth term \( \alpha g\circ W \). Convergence of \eqref{eq:pg_method} to a solution of \eqref{eq:problem_leastsquares} is guaranteed whenever \( f \) has an \( L \)-Lipschitz continuous gradient and \( \alpha \) is chosen smaller than \( 2/L \) (see, e.g., \cite{Combettes-Wajs-2005}).

However, proximal gradient methods suffer from two main drawbacks. The first is that scheme \eqref{eq:pg_method} assumes that \( \operatorname{prox}_{\alpha g \circ W} \) is available in closed form. This assumption can be restrictive in real-world applications, where it is often not available in closed form, as is the case, for example, of the Total Variation \cite{Bach-2012,Polson-et-al-2015}. To overcome this issue, splitting approaches such as primal-dual methods avoid the explicit computation of \( \operatorname{prox}_{\alpha g \circ W} \) by reformulating \eqref{eq:problem_leastsquares} as a convex-concave saddle-point problem \cite{Chambolle10,Malitsky-Pock-2018,Chambolle-et-al-2024}. Furthermore, if the convex conjugate \( g^*(v) = \sup_{w\in\mathbb{R}^{d'}}\langle v,w\rangle - g(w) \) admits an easily computable proximal operator, then the desired proximal approximation can be obtained by means of a \textit{primal-dual} inner routine \cite{Bonettini-et-al-2016a,Schmidt2011,Villa-etal-2013}, involving only the computation of \( \nabla f(u^k) \), \( \operatorname{prox}_{g^*} \), and matrix-vector products with $W$ and $W^T$ \cite{Bonettini2018a,Chen-et-al-2018,Villa-etal-2013}.

The second drawback is that, if the step length \( \alpha \) is chosen too small, convergence may become very slow, especially when only a rough estimate of the Lipschitz constant \( L \) is available. One possible remedy is to accelerate the method either by adopting a variable-metric strategy, which incorporates second-order information from the differentiable part \cite{Bonettini-et-al-2016a,Chouzenoux-etal-2014,Frankel-etal-2015,Ghanbari-2018,Lee2018}, or by introducing an extrapolation step that exploits information from previous iterations \cite{Beck-Teboulle-2009a,Ochs-etal-2014}. For instance, when \( g \) is the \( \ell^1 \)-norm, a well-known extrapolated proximal gradient method is the Fast Iterative Soft Thresholding Algorithm (FISTA) \cite{FISTA}.
Another strategy that has proven effective in accelerating convergence is the use of preconditioning techniques \cite{ITTA,PNPD,NPDIT}. In imaging problems, acceleration can often be achieved by preconditioning the linear system associated with \eqref{eq:Model_Equation} by means of a suitable operator $P$. In particular, when a left preconditioning approach is considered, the resulting method is also known as Iterated Tikhonov Thresholding Algorithm (ITTA). More recently, \cite{MM_FISTA} introduced a multilevel strategy as an alternative approach for accelerating proximal gradient methods. The main idea is to approximate the objective function by a sequence of lower-dimensional models defined on coarse scales, where descent steps can be computed at a lower cost and then transferred back to the fine level. These coarse corrections are incorporated into the fine-level descent step before the proximal operator is approximated to combine the benefits of both inertial and multilevel acceleration.

The combination of the two acceleration strategies based on extrapolation and preconditioning may lead to instability, particularly when the proximity operator is not available in closed form and must be approximated, see \cite{PNPD}. On the other hand, combining preconditioning with multilevel strategies can accelerate convergence while preserving the stability of the method \cite{buccini2020multigrid}.
Therefore, in this manuscript, we aim to further investigate acceleration techniques by combining the multilevel strategy introduced in \cite{MM_FISTA,IML_FISTA} with the preconditioning approach proposed in \cite{ITTA,PNPD}. 

Since our focus is on image deblurring problems, the variational problem \eqref{eq:problem_leastsquares} can be naturally defined also on coarser levels by exploiting numerical linear algebra tools to construct both the blurring operator $A$ and the preconditioner $P$ on the corresponding coarse grids. 
By choosing a suitable preconditioner $P$, we can eliminate the need for an extrapolation step not only at the finest level but also on the coarser ones. 
The use of the preconditioner also at the coarse levels allows us to compute more accurate coarse corrections with fewer iterations than in the standard multilevel strategy, where non-preconditioned methods are employed. As a consequence, the descent directions generated on the coarse grids are expected to be more effective once transferred back to the fine level. 
Thereby reducing both the number of iterations required for convergence and the overall computational time.

Several multilevel methods for image deblurring have previously been proposed in the literature \cite{donatelli2005multigrid,donatelli2006regularizing,morigi2008cascadic,chan2010multilevel,buccini2020multigrid}. Some of these methods share our choice of grid-transfer and coarse-level operators; however, the underlying iterative regularization schemes are not based on proximal gradient methods. In comparison with the two papers \cite{MM_FISTA,IML_FISTA} that inspired our work, we not only consider different proximal gradient methods, but also introduce a different correction criterion and an adaptive strategy for selecting the number of coarse-level iterations and multilevel cycles, thereby making the proposed multilevel framework fully automated.

The remainder of the manuscript is organized as follows. \Cref{sec:prel} reviews standard notions and results from convex analysis that will be used throughout the paper. \Cref{sec:prox_methods} presents the proximal gradient methods considered in this work, distinguishing between those for which the proximal operator is available in closed form and those in which an approximate proximal point is computed via a dual formulation. \Cref{sec:multilevel} introduces the multilevel framework and describes the strategies adopted in our approach. The convergence properties of the resulting methods are analyzed in \Cref{sec:conv}. Finally, \Cref{sec:exp_set,sec:num_res} are devoted to the experimental setup and the numerical results for image deblurring problems.


\section{Preliminaries}\label{sec:prel}
This section recalls some well-known definitions and properties of convex analysis that will be used throughout the paper. For a more in-depth discussion, we refer the reader to \cite{rockafellar1997convex}. Before that, we introduce some useful notations that will be consistently used.

Denoting by 
$\mathcal{S}_+(\R^d)$ the set of all $d\times d$ real symmetric positive definite matrices and given $\eta>0$, we define the subset $\mathcal{D}_{\eta}\subseteq \mathcal{S}_+(\R^{d})$ as containing all matrices in $\mathcal{S}_+(\R^{d})$ with eigenvalues greater than or equal to $\eta$. Moreover, given $M\in \mathcal{S}_+(\R^d)$, the norm induced by $M$ is defined as
$\|u\|_{M} = \sqrt{u^TMu}$, $\forall u\in\R^d$.
\subsection{Convex Analysis}
\begin{definition}[Relative interior]
    Let $\Omega\subseteq\mathbb{R}^{n}$ be a set. We define the \textit{relative interior} of $\Omega$ as
    \begin{equation*}
        \operatorname{relint}(\Omega) := \left\{u\in \Omega \ |\ \exists\ \epsilon > 0 \text{\quad s.t.\quad } B(u,\epsilon)\cap\operatorname{aff(\Omega)}\subset \Omega\right\},
    \end{equation*}
    where $B(u,\epsilon)$ is the ball centered in $u$ with radius $\epsilon$ while $\operatorname{aff}(\Omega)$ is the affine hull of $\Omega$.
\end{definition}

\begin{definition}[Effective domain]
    The \textit{effective domain} of a proper function $f\colon\mathbb{R}^d\to\mathbb{R}\cup\{\infty\}$ is the set
    \begin{equation*}
        \dom(f) := \left\{u\in\mathbb{R}^{d}:f(u)<\infty\right\}.
    \end{equation*}
    Note that since $f$ is proper, there exists $u\in\mathbb{R}^d$ such that $f(u)\ne+\infty$ and thus $\dom(f)\ne\emptyset$.
\end{definition}
\begin{definition}
    A function $f:\R^d\rightarrow \R$ is said to have an $L-$Lipschitz continuous gradient if the following property holds
    \begin{equation*}
        \|\nabla f(\tilde{u})-\nabla f(\bar{u})\|\leq L\|\tilde{u}-\bar{u}\|, \quad \forall \ \tilde{u},\bar{u}\in\R^d.
    \end{equation*}
\end{definition}

\begin{example}
A classic example of a function with $L-$Lipschitz continuous gradient is the least squares data-fidelity term \eqref{eq:fLS}.
Since
$
    \nabla f(u) = A^T(Au - b^\delta),
$
we have 
\begin{flalign*}
    \|\nabla f (\tilde{u}) - \nabla f(\bar{u})\| &= \|A^T(A\tilde{u} - b^\delta) - A^T(A\bar{u} - b^\delta)\|\leq \|A^T A\| \|\tilde{u} - \bar{u}\|,
\end{flalign*}
and, in particular, the Lipschitz constant of $\nabla f$ is $L = \|A^T A\|=\|A\|^2$.
\end{example}

\begin{definition}
    Let $\varphi:\R^d\rightarrow\R\cup\{\infty\}$ be a proper convex and lower semicontinuous function. The subdifferential of $\varphi$ at point $u\in\R^d$ is defined as the set
    \begin{equation*}
        \partial\varphi = \{w\in\R^d: \varphi(v)\ge\varphi(u) + \left \langle w , v-u\right \rangle, \quad \forall v\in\R^d\}.
    \end{equation*}
\end{definition}

\begin{lemma}\cite[Theorem 23.8, Theorem 23.9]{rockafellar1997convex} 
    Let $\varphi_1,\varphi_2 : \R^{d'}\rightarrow\R\cup\{\infty\}$ be proper, convex and lower semicontinuous functions and $W\in\R^{d'\times d}$.
    \begin{itemize}
        \item If $\varphi(v) = \varphi_1(v)+\varphi_2(v)$ and there exists $v_0\in\R^{d'}$ such that $v_0\in \operatorname{relint}(\operatorname{dom}(\varphi_1))\cap\operatorname{relint}(\operatorname{dom}(\varphi_2))$, then
              \begin{equation*}
                  \partial(\varphi_1+\varphi_2)(v) = \partial\varphi_1(v)+\partial\varphi_2(v), \quad \forall v\in\operatorname{dom}(\varphi).
              \end{equation*}
        \item If $\varphi = \varphi_1(Wu)$ and there exist $u_0\in\R^d$ such that $Wu_0\in\operatorname{relint}(\operatorname{dom}(\varphi_1))$, then
              \begin{equation*}
                  \partial\varphi(u) = W^T\partial\varphi_1(Wu) =\{W^Tv : v\in\partial\varphi_1(Wu)\}, \quad \forall u\in\operatorname{dom}(\varphi).
              \end{equation*}
    \end{itemize}
\end{lemma}

\begin{definition}\cite[p. 104]{rockafellar1997convex}\label{def:convex_conj}
    Given a proper, convex, lower semicontinuous function $\varphi:\mathbb{R}^d\rightarrow \mathbb{R}\cup\{\infty\}$, the convex conjugate of $\varphi$ is the function
    \begin{equation*}
        \varphi^*:\R^d\rightarrow \R\cup\{\infty\}, \quad \varphi^*(w)=\underset{u\in\R^d}{\operatorname{\sup}}\ \langle w,u\rangle - \varphi(u), \ \forall \ w\in\R^d.
    \end{equation*}
\end{definition}
For the biconjugate function $(\varphi^*)^*$, the following well-known result holds.
\begin{theorem}\cite[Theorem 12.2]{rockafellar1997convex}\label{lem:biconjugate}
    Let $\varphi:\R^d\rightarrow \R\cup\{\infty\}$ be proper, convex, and lower semicontinuous. Then $\varphi^*$ is convex and lower semicontinuous, and $(\varphi^*)^*=\varphi$.
\end{theorem}

\begin{definition}\label{def:proxP}
    Let $\varphi\colon\mathbb{R}^d\rightarrow\R\cup\{+\infty\}$ be a proper, convex, lower semicontinuous function. The proximity operator associated to $\varphi$ in the metric induced by a symmetric positive definite matrix $P\in\mathcal{D}_{\eta}$ with parameter $\alpha>0$ is the map $\mathrm{prox}_{\alpha \varphi}:\mathbb{R}^d\rightarrow\mathbb{R}^d$ defined as
    $$
        \mathrm{prox}_{\alpha \varphi}^{P}(a) = \arg\min_{u\in\mathbb{R}^d}\varphi(u)+\frac{1}{2\alpha}\|u-a\|^2_P, \quad \forall \ a\in\mathbb{R}^d.
    $$
\end{definition}
\begin{remark}
When $P=I_n$, we write $\prox_{\alpha\varphi}^{I_n}=\prox_{\alpha\varphi}$.
\end{remark}
\begin{proposition}[Moreau decomposition]
Let $\varphi\colon\mathbb{R}^d \to \mathbb{R}\cup\{\infty\}$ be a proper, convex, lower semicontinuous function, and
let $\varphi^{\ast}:\mathbb{R}^d \to \mathbb{R}\cup\{\infty\}$ be its conjugate.
Given $\alpha \in \mathbb{R}_{++}$,
and $P\in\mathcal{S}_+(\mathbb{R}^d)$, then the following identity holds
\begin{equation}\label{prep:Moreau_dec}
    \operatorname{prox}_{\alpha \varphi}^{P}(u)+\alpha P^{-1}\,
\operatorname{prox}_{\alpha^{-1} \varphi^{\ast}}^{P^{-1}}
\bigl(\alpha^{-1} P u\bigr)=u,
\qquad \forall u \in \mathbb{R}^d.
\end{equation}
\end{proposition}

\begin{lemma}\cite[Section 11--12]{Moreau_fr}
    Let $\varphi:\R^d\rightarrow\R\cup\{\infty\}$ be proper, convex, and lower semicontinuous. For all $\alpha,\beta>0$, the following statements are equivalent:
    \begin{enumerate}
        \item u = $\prox_{\alpha\varphi}(u+\alpha w)$;
        \item $w\in\partial\varphi(u)$;
        \item $\varphi(u)+\varphi^*(w) = \left \langle w, u\right \rangle$;
        \item $u\in\partial\varphi^*(w)$;
        \item $w = \prox_{\beta\varphi^*}(\beta u+w)$.
    \end{enumerate}
\end{lemma}

\begin{definition}{\emph{(Smoothed convex function)}} \label{def:smooth_convex_func}
Let $R$ be a convex, lower semi-continuous, and proper function on $\mathbb{R}^{d}$. For every $\gamma>0$, a continuously differentiable function $R_{\gamma}$ is called a smoothed convex approximation of $R$ if there exist two finite-valued scalars $\eta_1$ and $\eta_2$ satisfying $\eta_1+\eta_2 >0$ such that it holds
\begin{equation}
    \forall u \in \mathbb{R}^{d} \qquad R(u) - \eta_1 \gamma \leq R_{\gamma}(u) \leq R(u) + \eta_2 \gamma.
\end{equation}
\end{definition}
A classical and widely adopted approach to constructing such a smoothed approximation is via the Moreau envelope, which we define next.
\begin{definition}[Moreau envelope]\label{def: moreau}
Let $\gamma > 0$ and $\varphi\colon  \mathbb{R}^d \to \mathbb{R}\cup\{\infty\}$ be a convex, lower semi-continuous, and proper function. The Moreau envelope of $\varphi$, denoted by $^\gamma \varphi$, is the convex, continuous, real-valued function defined as
\begin{equation}
^\gamma \varphi(u) = \inf_{v \in \mathbb{R}^d} \left\{ \varphi(v) + \frac{1}{2\gamma} \|u - v\|^2 \right\}.
\end{equation}
\end{definition}

The Moreau envelope $^\gamma \varphi$ is strictly related to $\mathrm{prox}_{\gamma \varphi}$ by
\begin{equation}
^\gamma \varphi(u) = \varphi(\mathrm{prox}_{\gamma \varphi}(u)) + \frac{1}{2\gamma} \|u - \mathrm{prox}_{\gamma \varphi}(u)\|^2.
\end{equation}
Moreover, $^\gamma \varphi$ is Fr\'echet differentiable on $\mathbb{R}^d$, and its gradient is $\gamma^{-1}$-Lipschitz continuous and it holds that 

\begin{equation}\label{eq: grad_smooth}
\nabla (^\gamma \varphi) = \gamma^{-1} \big(\mathrm{Id} - \mathrm{prox}_{\gamma \varphi}\big).
\end{equation}
In light of these regularizing properties, we will employ the Moreau envelope as our standard smoothed convex approximation throughout this work.

From now on, we consider the optimization problem \eqref{eq:problem_leastsquares} under the following 
\begin{assumption}
    \label{ass: standard}
    \quad
    \begin{enumerate}
        \item $f: \label{hyp:f} \mathbb{R}^d \to \mathbb{R}$ is convex and differentiable with $L$-Lipschitz continuous gradient;
        \item \label{hyp: standard_h} $g: \mathbb{R}^{d'} \to \mathbb{R}\cup\{\infty\}$ is a proper convex lower semicontinuous function;
        \item \label{hyp: standard_W} $W\in \mathbb{R}^{d'\times d}$ and exists $u_0$ such that $Wu_0 \in \relint(\dom(g))$;
        \item Problem in equation \eqref{eq:problem_leastsquares} has at least one solution.
    \end{enumerate}
\end{assumption}

We remark that the function $f(u)$ in \eqref{eq:fLS} satisfies Assumption~\ref{ass: standard}.\ref{hyp:f}. Moreover, Assumption~\ref{ass: standard}.\ref{hyp: standard_W} on $W$ is required to guarantee that the subdifferential rule $\partial (g \circ W)(u) = W^T \partial g(Wu)$ holds.

\section{Proximal gradient methods}\label{sec:prox_methods}
In this section, we review different proximal gradient or forward-backward methods that can effectively exploit the structure of the variational problem \eqref{eq:problem_leastsquares}.  
These are first-order iterative methods that alternate, at each iteration, a forward gradient step on the differentiable part $f$, followed by a backward proximal step on the convex, non-smooth term $g\circ W$. The general iterative scheme of these methods is given by
\begin{equation}\label{eq:fbit}
    u^{k+1}= \prox_{\alpha_kg\circ W}(u^k-\alpha_k\nabla f(u^k)), \qquad k=1,2,\ldots,
\end{equation}
where $\alpha_k\in\R_{++}$ is the steplength parameter.
In the first part of this section, we assume that the proximal operator of $g\circ W$ is available in closed form. In contrast, the second part is devoted to numerical methods that address the inexact computation of the proximal operator.

\subsection{Exact proximal gradient methods}
A particular instance of proximal gradient methods appears when we consider $g(u) = \lambda \|u\|_1$, $\lambda >0$, and $W$ is a semiorthogonal linear operator, i.e., $W^TW=I$, for instance, a framelet operator (see Appendix~\ref{app:framelets}). In this case, the proximity operator of $\lambda \|\cdot\|_1\circ W$ is the combination of the soft-thresholding operator with the linear operator $W$, namely
\begin{equation}
    \prox_{\lambda \|\cdot\|_1\circ W}(u)=W^T\mathcal{S}_{\lambda}(Wu),
\end{equation}
where $(\mathcal{S}_{\lambda}(u))_i \coloneqq \operatorname{sign}(u_i)\max\{|u_i|-\lambda,0\}$, for $i=1,\ldots d$, is the soft--thresholding operator. When we apply the proximal gradient method to this specific case, we obtain the so-called \textit{Iterative Soft--Thresholding Algorithm} (ISTA) \cite{ISTA}. 

\paragraph{FISTA}
Due to the slow convergence rate typical of first-order methods, a Nesterov-type acceleration strategy can be incorporated into the iterations. This consists of adding a preliminary step, called the extrapolation step, to the iterative scheme under consideration. The resulting \textit{Fast ISTA} (FISTA) \cite{FISTA} is described in \Cref{alg:fista}.

\begin{algorithm}[H]
\caption{FISTA} \label{alg:fista}
\begin{algorithmic}[1]
\State Choose \( u^0 \in \mathbb{R}^d \), \( 0<\alpha<\frac{2}{L} \), \( {\mathrm{maxit}} \in \mathbb{N} \), $W\in \mathbb{R}^{d'\times d}$.
\State $t_0 = 1$, $u^{-1} = u^{0}$
\For{$k=1,\dots,$ maxit}
    \State $t_{k+1} = \dfrac{1 + \sqrt{1 + 4 t_k^2}}{2}$
    \State $y^{k} = u^{k} + \dfrac{t_k - 1}{t_{k+1}} (u^{k} - u^{k-1}) $
    \State $u^{k+1} = W^T\mathcal{S}_{\alpha\lambda} \big(W(y^k - \alpha \nabla f(y^k))\big)$ \label{step: FISTA_eval}
\EndFor
\end{algorithmic}
\end{algorithm}

\paragraph{ITTA}
A different approach to overcoming the slow convergence of proximal gradient methods relies on introducing a suitable preconditioner 
$P\in\mathcal{S}_+(\mathbb{R}^d)$. 
This strategy is inspired by the iterated Tikhonov method \cite{engl1996regularization}, which can also be interpreted as a preconditioned gradient descent method and is therefore referred to as the \textit{Iterated Tikhonov Thresholding Algorithm} (ITTA) in \cite{ITTA}.

Similarly to equation \eqref{eq:problem_leastsquares}, 
given $S\in \mathcal{S}_+(\mathbb{R}^s)$, we consider the optimization problem
\begin{align}
    \label{eq:model_PNPD}
    \begin{split}
        \argmin_{u\in\mathbb{R}^d} f_S(u) + g(Wu) = \argmin_{u\in\mathbb{R}^d}\frac{1}{2} \|Au-b^{\delta}\|_{S^{-1}}^2 +g(Wu).
    \end{split}
\end{align}
In the data fidelity term
\begin{equation}\label{eq:fS}
    f_S(u) = \frac{1}{2}\|S^{-\frac{1}{2}}(Au-b^{\delta})\|^2,
\end{equation}
the linear operator $S^{\frac{1}{2}}$ can be interpreted as a left preconditioner for the linear system~\eqref{eq:Model_Equation}. If we further assume that there exists $P\in \mathcal{S}_+(\R^d)$ such that
\begin{equation}\label{eq:commutative_PNPD}
    P^{-1}A^T = A^T S^{-1},
\end{equation}
then it holds
\begin{equation}\label{eq:GfS}
    \nabla f_S(u) = A^TS^{-1}(Au-b^{\delta}) = P^{-1}\nabla f(u),
\end{equation}
where $f(u)$ is defined as in equation \eqref{eq:fLS}.

A possible choice for the preconditioner operator $P$  and the associated matrix $S$ that satisfies the condition \eqref{eq:commutative_PNPD}, is
\begin{equation}
    \label{eq:pnpd_P_numerical}
    P = A^T A + \nu I, \qquad S = A A^T + \nu I,
\end{equation}
with $\nu > 0$, as proposed in \cite{PNPD}. This choice is inspired by the iterated Tikhonov method \cite{engl1996regularization}, which coincides with the Levenberg–Marquardt method applied to the minimization of $f$. Indeed, $P$ is a regularized approximation of the Hessian of $f$ in $\mathcal{S}_+(\mathbb{R}^d)$. 

Applying the forward-backward method \eqref{eq:fbit} to \eqref{eq:model_PNPD}, where $f$ is replaced by $f_S$, and choosing $P$ as in \eqref{eq:pnpd_P_numerical}, such that we can fix $\alpha_k = 1$, the corresponding forward-backward method can be summarized as in \Cref{alg:itta}.

\begin{algorithm}[H]
\caption{ITTA} \label{alg:itta}
\begin{algorithmic}[1]
\State Choose \( u^0 \in \mathbb{R}^d \), \( {\mathrm{maxit}} \in \mathbb{N} \), $W\in \mathbb{R}^{d'\times d}$. 
\For{$k=1,\dots,$ maxit}
    \State $u^{k+1} = W^T\mathcal{S}_{\lambda} \big(W(u^k - P^{-1}\nabla f(u^k) )\big)$ \label{step: ITTA_eval}
\EndFor
\end{algorithmic}
\end{algorithm}
\subsection{Inexact proximal gradient methods}\label{sec:inex_prox_methods}
The availability of a closed form for the proximity operator of \( g\circ W \) is a strong assumption that is not satisfied by many regularization terms. In what follows, we describe a possible remedy for this issue by introducing a primal-dual strategy. Furthermore, we briefly present a preconditioned iteration, which is a generalization of the ITTA scheme within a broader framework.

\paragraph{NPD}
Several iterative methods based on an inexact inertial forward–backward splitting scheme have been proposed in the literature, for instance in \cite{Schmidt2011, Villa-etal-2013, Bonettini2018a, NPD}. In this work, we focus on the \textit{Nested Primal--Dual} (NPD) method proposed in \cite{NPD}, in which a nested iterative procedure is used to approximate the proximity operator of $g \circ W$. Its iterative scheme can be summarized as 
\begin{align}
    \label{eq:NPD}
    \begin{cases}
        \bar{u}^k = u^k + \gamma_k (u^k - u^{k-1}), \\
        u^{k+1} \approx \prox_{\alpha_k g\circ W}(\bar{u}^k - \alpha_k \nabla f(\bar{u}^k)),
    \end{cases}
\end{align}
where \( \approx \) indicates an approximation of the proximal-gradient point. Here, \( \bar{u}^k \) is usually referred to as the inertial point and \( \gamma_k \in [0,1] \) is the inertial parameter.
Let  \( L \) be the Lipschitz constant of \( \nabla f \), then the step length $\alpha_k$ can be chosen stationary as $\alpha \in (0, 2/L)$.

We can compute a fairly accurate approximation of the proximity operator in \eqref{eq:NPD} by exploiting the following

\begin{lemma}\cite[Theorem 1]{NPD}\label{lemma:primal_dual_NPD}
    Suppose that \( g: \mathbb{R}^{d'} \rightarrow \mathbb{R} \cup \{\infty\} \) and \( W \in \mathbb{R}^{d' \times d} \) satisfy Assumption~\ref{ass: standard} items 2. and 3. Given \( a \in \mathbb{R}^d \), \( \alpha > 0 \), \( 0 < \beta < 2/\|W\|^2 \), and \( v^0 \in \mathbb{R}^{d'} \), define the sequence
    \begin{equation}\label{eq:dual_seq_NPD}
        v^{k+1} = \prox_{\beta \alpha^{-1}g^*}(v^k+\beta\alpha^{-1}W(a-\alpha W^Tv^k)), \quad \forall k \geq 0.
    \end{equation}
    Then there exists \(\hat{v}\in\mathbb{R}^{d'}\) such that:
    \begin{itemize}
        \item[(i)] \(\lim_{k\rightarrow\infty}v^{k} = \hat{v}\);
        \item[(ii)] \(\prox_{\alpha g\circ W}(a) = a - \alpha W^T\hat{v}.\)
    \end{itemize}
\end{lemma}

Note that by combining items $(i)$ and $(ii)$ of Lemma~\ref{lemma:primal_dual_NPD} with equation \eqref{eq:dual_seq_NPD}, it holds
\begin{equation*}
    \prox_{\alpha g\circ W}(a) = a-\alpha W^T\lim_{k\rightarrow\infty}\prox_{\beta \alpha^{-1}g^*}(v^k+\beta\alpha^{-1}W(a-\alpha W^Tv^k)).
\end{equation*}
In this way, we can compute an approximation of the proximal operator of \( g\circ W \) by a finite number of steps of a primal-dual procedure, provided that the operator \(\prox_{\beta\alpha^{-1} g^*}\) and the matrix-vector product with the linear operators \( W \) and \( W^T \) are easily computable. 
Let \( j_{\mathrm{max}} \) be the maximum number of inner iterations used to compute the dual sequence \eqref{eq:dual_seq_NPD}, the resulting NPD algorithm is summarized in Algorithm~\ref{alg:npd}.

\begin{algorithm}
    \caption{NPD}
    \label{alg:npd}
    \begin{algorithmic}[1]
        \State Choose \( u^0 \in \mathbb{R}^d \), \( 0<\alpha<\frac{2}{L} \), \( 0<\beta<\frac{2}{\|W\|^2} \), \( j_{\mathrm{max}} \in \mathbb{N} \), \(\{\gamma_k\}_{k\in \mathbb{N}} \subseteq \mathbb{R}_{\ge 0}\).
        \State \( u^{-1} = u^0 \)
        \State \( v^{0,0} = 0 \)
        \For{\( k= 0, 1, \dots \)}
        \State Compute \(\gamma_k\)
        \State \(\bar{u}^k = u^k + \gamma_k (u^k - u^{k-1})\)
        \State \( u^{k+\frac{1}{2}} = \bar{u}^k - \alpha \nabla f(\bar{u}^k) \)
        \For{\( j = 0, 1, \dots, j_{\text{max}}-1 \)}
        \State \( u^{k,j} = u^{k+\frac{1}{2}} - \alpha W^T v^{k,j} \)
        \State \( v^{k,j+1} = \prox_{\beta \alpha^{-1} g^*}(v^{k,j} + \beta \alpha^{-1} W u^{k,j}) \)
        \EndFor
        \State \( v^{k+1,0} = v^{k, j_{\text{max}}} \)
        \State \( u^{k,j_{\text{max}}} = u^{k+\frac{1}{2}} - \alpha W^T v^{k,j_{\text{max}}} \)
        \State \( u^{k+1} = \frac{1}{j_{\text{max}}} \sum_{i=1}^{j_{\text{max}}} u^{k,i} \)
        \EndFor
    \end{algorithmic}
\end{algorithm}

The convergence of the iterations to a solution of the variational problem \eqref{eq:problem_leastsquares} can be proved under suitable assumptions on the extrapolation parameter $\gamma_k$, provided that \Cref{ass: standard} holds. For further details, we refer the reader to \cite{NPD}.

\paragraph{PNPD}\label{eq:secright}

In \cite{PNPD}, the authors prove that preconditioning strategies can be applied to the NPD framework to achieve faster convergence. 
In what follows, we consider only left preconditioning; for further details on the right preconditioning approach, the reader is referred to \cite{PNPD, NPDIT}.

The main idea is similar to the one used to derive ITTA from ISTA.
Therefore, by applying the NPD algorithm to problem~\eqref{eq:model_PNPD} instead of problem~\eqref{eq:problem_leastsquares} under assumption~\eqref{eq:commutative_PNPD}, we obtain a two-step iterative scheme, called the \textit{Preconditioned Nested Primal-Dual} (PNPD) \cite{PNPD}, defined as
\begin{equation}
    \label{eq:PNPD_scheme}
    \begin{cases}
        \bar{u}^k= u^k + \gamma_k(u^k-u^{k-1}), \\
        u^{k+1} \approx \prox_{\alpha_k g\circ W}(\bar u^k - \alpha_k P^{-1} \nabla f(\bar u^k)).
    \end{cases}
\end{equation}
The preconditioner $P$ and the extrapolation strategy are both intended to accelerate the iterative scheme. Indeed, as shown in \cite{PNPD}, the acceleration provided by the extrapolation term can also be achieved through a suitable choice of the preconditioner, and an inaccurate choice of $\gamma_k$ could lead to instability. Therefore, in the numerical experiments we set $\gamma_k=0$.

Algorithm~\ref{alg:pnpd} reports the pseudocode of the proposed PNPD method without extrapolation, where $\alpha_k = \alpha \in (0, 2/L)$ and \( L_S \) denotes the Lipschitz constant of \( \nabla f_S \). This constant depends on the choice of the preconditioner \( S \) (i.e., \( P \)), as a consequence of \eqref{eq:commutative_PNPD}.
In particular, thanks to equation~\eqref{eq:GfS},
we have
    \begin{equation*}
        \| \nabla f_S(x) - \nabla f_S(y) \| =  \| P^{-1} \nabla f(x) - P^{-1} \nabla f(y) \|
        \leq \|P^{-1} A^T A \| \| x - y \|,
    \end{equation*}
and it holds
\[
L_S = \|P^{-1} A^T A \|= \| A^T S^{-1} A \|.
\]

\begin{remark}
If $P$ and $S$ are chosen as in \eqref{eq:pnpd_P_numerical}, then it can be shown, by means of the singular value decomposition (SVD) of $A$, that $L_S \leq 1$ for all $\nu > 0$. Therefore, the step length in \eqref{eq:PNPD_scheme} can be fixed as $\alpha_k = 1$ throughout the iterations.
\end{remark}

The convergence of the iterates generated by the PNPD scheme can be established under the same assumptions as in the NPD case, together with condition \eqref{eq:commutative_PNPD} on the preconditioner $P$. For further details, we refer the reader to \cite{PNPD}.

\begin{algorithm}
    \caption{PNPD}
    \label{alg:pnpd}
    \begin{algorithmic}[1]
        \State Choose \( u^0 \in \mathbb{R}^d \), \(
        P\in\mathcal{S}_+(\R^d) \), \( 0<\alpha<\frac{2}{L_S} \), \(
        0<\beta<\frac{2}{\|W\|^2} \), \( j_{\mathrm{max}} \in \mathbb{N} \).
        \State \( v^{0,0} = 0 \)
        \For{\( k = 0, 1, \dots \)}
        \State \( u^{k+\frac{1}{2}} = u^k - \alpha P^{-1}\nabla f(u^k) \)
        \For{\( j = 0, 1, \dots, j_{\text{max}}-1 \)}
        \State \( u^{k,j} = u^{k+\frac{1}{2}} - \alpha W^T v^{k,j} \)
        \State \( v^{k,j+1} = \prox_{\beta \alpha^{-1} g^*}(v^{k,j} + \beta \alpha^{-1} W u^{k,j}) \)
        \EndFor
        \State \( v^{k+1,0} = v_{n}^{j_{\text{max}}} \)
        \State \( u^{k,j_{\text{max}}} = u^{k+\frac{1}{2}} - \alpha W^T v^{k,j_{\text{max}}} \)
        \State \( u^{k+1} = \frac{1}{j_{\text{max}}} \sum_{i=1}^{j_{\text{max}}} u^{k,i} \)
        \EndFor
    \end{algorithmic}
\end{algorithm}

\section{Multilevel framework} \label{sec:multilevel}
In this section, we introduce the multilevel framework adopted in the numerical experiments. The proposed approach is designed to enhance computational efficiency by combining multilevel strategies with suitable preconditioning techniques. The interplay between these two components improves both the convergence properties and the scalability of the overall method. The adopted strategy follows the multilevel scheme proposed in \cite{IML_FISTA}, suitably adapted to the structure of the problem under consideration.

In what follows, we introduce all the tools required to implement an effective multilevel scheme, together with the theoretical assumptions needed to guarantee convergence to the solution.

First, for the sake of simplicity and without loss of generality, we consider the two-level case. Accordingly, we use the index $h$ (respectively, $H$) for all quantities defined at the fine (respectively, coarse) level. Referring to the variational problem~\eqref{eq:problem_leastsquares}, we define $F_h\colon \mathbb{R}^{N_h}\to(-\infty,+\infty]$ as the objective function at the fine level, where $N_h=d$, and such that 
\[F_h=f_h+g_h, \qquad f_h=f, \; g_h=g\circ W.\] 
We then associate with this fine-level functions a coarse-level approximation,  where $N_H<N_h$, denoted by $F_H\colon \mathbb{R}^{N_H}\to(-\infty,+\infty]$, $f_H$, and $g_H$, respectively.

\subsection{Coarse model and transfer information operators}\label{sec:coarse_model_trans_operators}
Firstly, a multilevel scheme requires the transfer of information from one level to another. To this end, we introduce two transfer operators. The first one is the linear operator $I^H_h : \mathbb{R}^{N_h} \rightarrow \mathbb{R}^{N_H}$, referred to as the \textit{restriction operator}, which transfers information from the fine level to the coarse level. Conversely, we define the operator $I^h_H : \mathbb{R}^{N_H} \rightarrow \mathbb{R}^{N_h}$, called the \textit{prolongation operator}, which transfers information from the coarse level back to the fine level. We require these transfer operators to satisfy the following definition.
\begin{definition}
Two operators $I^H_h : \mathbb{R}^{N_h} \rightarrow \mathbb{R}^{N_H}$ and $I^h_H : \mathbb{R}^{N_H} \rightarrow \mathbb{R}^{N_h}$ are called coherent information transfer (CIT) operators if there exists $\xi > 0$ such that
\begin{equation}
I^h_H = \xi (I_h^H)^T.
    \label{def:CIT}
\end{equation}
\end{definition}
There are many different ways to construct CIT operators, for instance, by means of full weighting and bilinear interpolation, as adopted in the numerical results.

The basic idea of a multilevel strategy is to compute inexpensive corrections at a coarse resolution, which can then be used to update the current iterate after being prolonged to the fine level. We now focus on how to define the variational model at a coarser level. \\
Following \cite{IML_FISTA,Beck2012Smoothing}, we smooth the non-differentiable term $g\circ W$ in the sense of Definition~\ref{def: moreau}, which satisfies \Cref{def:smooth_convex_func}. 

\begin{definition}[Coarse model $F_H$ for non-smooth functions]\label{def:coarse_model}
The coarse model $F_H$ is defined at the point $u_h \in \mathbb{R}^{N_h}$ as
\begin{equation}\label{eq:FH}
F_H = f_H + g_{H} + \langle v_H, \,\cdot\, \rangle,
\end{equation}
where
\begin{equation}
\label{eq:vH_def}
v_H = I_h^H \bigl( \nabla f_h(u_h) + \nabla ^{\gamma_h}g_{h}(u_h) \bigr)
      - \bigl( \nabla f_H(I_h^H u_h) + \nabla ^{\gamma_H}g_{H}(I_h^H u_h) \bigr).
\end{equation}
The two functions $^{\gamma_h}g_h$ and $^{\gamma_H}g_H$ are smoothed versions of $g_h$ and $g_H$,
respectively, and they satisfy~\Cref{def:coherence_condition} with smoothing parameters
$\gamma_h > 0$ and $\gamma_H > 0$, respectively.
The smoothed coarse model $^{\gamma_H}F_H$ is defined at the point $u_h \in \mathbb{R}^{N_h}$ as
\begin{equation}\label{eq:FH_smooth}
^{\gamma_H}F_H = f_H + ^{\gamma_H}g_{H} + \langle v_H, \,\cdot\, \rangle.
\end{equation}
\end{definition}

\begin{definition}[First order coherence]\label{def:coherence_condition}
The first-order coherence between the smoothed version of the objective function $F_h$ at the fine level and the coarse level objective function $F_H$ is verified in a neighborhood of $u_h$ if it holds
\begin{equation}
\label{eq:coherence}
\nabla^{\gamma_H} F_H(I_h^H u_h) = I_h^H \, \nabla ^{\gamma_h}F_{h}(u_h).
\end{equation}
\end{definition}

\begin{lemma}
\label{lemma:first_order_coherence}
If $F_H$ is given by~\Cref{def:coarse_model}, then it satisfies the first order coherence (Definition~\ref{def:coherence_condition}).
\begin{proof}
Considering the gradient of the coarse smoothed model $^{\gamma_H}F_{H}$ in Equation \eqref{eq:FH_smooth} and combining it with the definition of $v_H$ in Equation \eqref{eq:vH_def}, yields 
\begin{align*}
    \nabla ^{\gamma_H}F_{H}(I_h^H u_h) & = \nabla f_H(I_h^H u_h) + \nabla^{\gamma_H}g_{H}(I_h^H u_h) + v_H \\
      & = I_h^H \bigl( \nabla f_h(u_h) + \nabla ^{\gamma_H}g_{h}(u_h) \bigr)\\
      &= I_h^H \, \nabla ^{\gamma_H}F_h (u_h).
\end{align*}
\end{proof}
\end{lemma}

Once the coarse model has been defined, the next step in a multilevel strategy consists of computing a suitable solution of the coarse problem and using it to obtain a better approximation at the fine level. More specifically, we require that a decrease in the objective function at the coarser level also implies a decrease in the fine-level objective function when the coarse solution is interpolated onto the finer grid. This consistency between levels is ensured by the linear term $v_H$ defined in \eqref{eq:vH_def}, which allows the coarse model to satisfy the first-order coherence condition of \Cref{def:coherence_condition}.
The smoothing approach aims to preserve fidelity to the original non-smooth problem while introducing desirable properties into the coarse model.\\
Depending on the regularization term, two different cases may arise. If $g\circ W$ is smooth, one can simply apply gradient-based methods without requiring any smoothing in the coarse model, that is, $\gamma_h=\gamma_H=0$. This is also the case when $g\circ W$ is non-smooth, but its proximity operator can be computed in closed form, since approximate solutions of the coarse model can be computed by means of proximal gradient methods.\\
However, when the proximity operator cannot be computed in closed form, as in the case of TV regularization, the use of the Moreau envelope may lead to some difficulties. Indeed, in order to compute the linear term $v_H$ defined in \eqref{eq:coherence}, we need to evaluate $\nabla^\gamma g_h = \nabla^\gamma(g\circ W)$. By \eqref{eq: grad_smooth}, we have
\[
    \nabla^\gamma (g\circ W) = \gamma^{-1}(\text{Id}-\prox_{\gamma g\circ W}).
\]
Since $\prox_{\gamma g\circ W}$ does not admit a closed-form expression, we are unable to compute the linear term $v_H$. To overcome this issue, we adopt the same strategy proposed in \cite{IML_FISTA}: instead of directly applying the Moreau envelope to $g\circ W$, we first compute the Moreau envelope of $g$ and then compose it with $W$. This smoothing strategy satisfies the following interesting properties.

\begin{lemma}[\cite{IML_FISTA}, Lemma 3.2]\label{lemma:approx}
    $^{\gamma}g\circ W$ is a smoothed convex function approximating $g\circ W$ in the sense of Definition \ref{def:smooth_convex_func}.
\end{lemma}

\begin{lemma}\label{lemma:comp}
For any $u \in \mathbb{R}^d$, $W: \mathbb{R}^d \to \mathbb{R}^{d'}$ and $g: \mathbb{R}^{d'} \to \mathbb{R}$ a convex, lower semi-continuous, and proper function, we have
\begin{equation}
\label{eq:gradient}
\nabla (^\gamma g\circ W)(u) = \gamma^{-1} W^T \big( Wu - \mathrm{prox}_{\gamma g}(Wu) \big).
\end{equation}

\end{lemma}

This means that an explicit expression for $\mathrm{prox}_{\gamma g}$ is sufficient to characterize the gradient of ${}^\gamma (g\circ W)$. However, if no explicit information on the proximity operator of $g$ is available, but only on its convex dual function $g^*$, then we can exploit the Moreau decomposition \eqref{prep:Moreau_dec}, with $P=\mathrm{Id}$, and rewrite, according to Lemma~\ref{lemma:approx}, the gradient of ${}^\gamma (g\circ W)$ as
\begin{align*}
    \nabla(^\gamma g\circ W)(u)&= \gamma^{-1} W^* \big( Wu - \mathrm{prox}_{\gamma g}(Wu) \big)\\
    &= \gamma^{-1} W^* \big( Wu - (Wu-\gamma \operatorname{prox}_{\gamma^{-1}g^\ast}(\gamma^{-1}Wu) )\big)\\
    &=W^T\operatorname{prox}_{\gamma^{-1}g^\ast}(\gamma^{-1}Wu).
\end{align*}

Within the multilevel framework, fine-level iterations are enhanced by coarse-level corrections, which are computed through V-cycles (see Algorithms \ref{alg:Vcycle} and \ref{alg:multilevel}). Specifically, during a V-cycle, a sequence of iterations is performed at the coarse level. The difference between the final and initial coarse-level iterates, which acts as an error-correction term, is then prolonged to the fine level and used to update the current approximation. Between successive V-cycles, standard fine-level iterations may be performed. Finally, the multilevel acceleration is applied only during the first $p$ iterations, after which the algorithm proceeds using the standard method alone.
\begin{algorithm}
\caption{Vcycle\_iter\quad (2 levels)}\label{alg:Vcycle}
\begin{algorithmic}[1]
\State $u_H^{k,0} = I_h^H u_h^k$
\State $v_{H}^k = I_h^H \, \nabla ^\gamma F_{h}(u_h^k) - \nabla ^\gamma F_{H}(u_H^{k,0})$
\State $u_H^{k,m} = \Phi(u_H^{k,0}, m, v_H^k)$
\State $c_h^{k,m}= I_H^h\left(u_H^{k,m}-u_H^{k,0}\right)$
\State $\tau^{k}=5$ ,\qquad $t=0$
\While{$F_h(u_h^k+\tau^{k}c_h^{k,m})> F_h(u_h^k) + c\tau^{k}\nabla ^\gamma F_h^T(u_h^k)c_h^{k,m}\land t < \mathrm{max_{bt}}$}
\State $\tau^{k} = (0.85)^t\,\tau^{k}$,\quad $t = t+1$
\EndWhile
\State $u_{h}^{k+1} = u_h^k +\tau^{k}c_h^{k,m}$
\end{algorithmic}
\end{algorithm}

\begin{algorithm}
\caption{MULTILEVEL\quad (2 levels)}\label{alg:multilevel}
\begin{algorithmic}[1]
\State Set $u_{h}^0 \in \mathbb{R}^N$, $p$, $\mathrm{max_{bt}}, m, \Phi, \mathrm{maxit}$.
\For {$k=1,\dots,p$}
\State $u_{h}^{k}=$\textbf{Vcycle\_iter($u_h^{k-1},m,\Phi_H,\mathrm{max_{bt}}$)}
\EndFor\\\vspace*{-0.5cm}
\State $u_{h}^{\mathrm{maxit}} = \Phi_{h}(u_h^p, \mathrm{maxit}-p)$
\end{algorithmic}
\end{algorithm}

\subsection{Decrease of the objective function}\label{sec:decrease_obj_func}
The last ingredient needed to properly define a multilevel strategy is the choice of the solver $\Phi_H$ to be used for the coarse variational problem. Ideally, this solver should generate a sequence $\{u_{H}^m\}_{m \geq 0}$ that yields a sufficient decrease of ${}^{\gamma_H}F_H$ after a finite number of iterations.
\begin{assumption}\label{ass:coarse_operator}
Let $(\Phi_{H,\ell})_{\ell \in \mathbb{N}}$ be a sequence of operators such that there exists an integer $m>0$ with the following property: if $u_{H}^m = \Phi_{H,m-1} \circ \ldots \circ \Phi_{H,0}(u_{H}^0)$,
then
\[
{}^{\gamma_H}F_H(u_{H}^m) \leq {}^{\gamma_H}F_H(u_{H}^0).
\]
Moreover, the correction $u_{H}^m-u_{H}^0$ is bounded.
\end{assumption}

Typical choices for the sequence of operators $\{\Phi_{H,\ell}\}_{\ell\in\mathbb{N}}$ include gradient descent methods and forward--backward schemes. For simplicity, in the numerical experiments section, we choose $\Phi_{H,\ell}$ to be the same solver used at the fine level. These solver operators guarantee that $u_{H}^m-u_{H}^0$ is a bounded descent direction for ${}^{\gamma_H}F_H$ (due to convergence of the sequence \cite{Dossal2015stability}).

\begin{lemma}{\emph{(Descent direction for the fine-level smoothed function)}.}\label{lemma:decrease_smooth}
Let us assume that $I_h^H$ and $I_H^h$ are CIT operators, that ${}^{\gamma_H}F_H$ satisfies Definition~\ref{def:coarse_model}, and that $\Phi_{H}$ satisfies Assumption~\ref{ass:coarse_operator}. Then, $I_H^h(u_{H}^m-u_{H}^0)$ is a descent direction for $f_h + {}^{\gamma_h}g_{h}$.
\begin{proof}
Let $u_h \in \mathbb{R}^{N_h}$ and define $p_H:=u_{H}^m-u_{H}^0$. Recall that $u_{H}^0 = I_h^H u_h$. By the definition of descent direction, we have
\begin{equation*}
     \langle p_H,\nabla F_{H}(u_{H}^0) \rangle \leq 0.
\end{equation*}
By the first-order coherence condition and imposing $I_h^H = \xi^{-1} \left(I_H^h\right)^T$, we obtain
\[\langle p_H,\nabla F_{H}(u_{H}^0) \rangle = \langle p_H,I_h^H \nabla (f_h + {}^{\gamma_h}g_{h})(u_h) \rangle =
    \xi^{-1} \langle I_H^h(p_H),\nabla (f_h + {}^{\gamma_h}g_{h})(u_h) \rangle \leq 0.\]
\end{proof}
\end{lemma}

The previous lemma shows that $I_H^h(u_{H}^m-u_{H}^0)$ is a descent direction for the smoothed objective function ${}^{\gamma_h}F_h$ at the fine level. We conclude this part on the decrease of the objective function by deriving a bound for the decrease of the non-smooth objective function $F_h$ at the fine level. In \cite{MM_FISTA}, the authors search for a suitable step size $\tau$ that prevents excessively large corrections from the coarse level by guaranteeing that
\begin{equation}
    {}^\gamma F_h(u_h + \tau I_H^h(u_{H}^m-u_{H}^0)) \leq {}^\gamma F_h(u_h).
\label{eq:decrease_smooth}
\end{equation}
In our case, however, we replace condition \eqref{eq:decrease_smooth} with the more restrictive Armijo condition. More precisely, we require the parameter $\tau$ to satisfy
\begin{equation}\label{eq:cond:Armijo}
    {}^\gamma F_h(u_h + \tau I_H^h(u_{H}^m-u_{H}^0)) \leq {}^\gamma F_h(u_h)+c_1\tau\langle\nabla{}^\gamma
F_h(u_h),I_H^h(u_{H}^m-u_{H}^0)\rangle,
\end{equation}
for some positive constant $c_1\in(0,1)$.

\begin{lemma}[Fine-level decrease]
\label{lemma:decrease_finefunction}
If the assumptions of Lemma~\ref{lemma:decrease_smooth} hold, then the iterations of Algorithm \ref{alg:multilevel_auto} satisfy
    \begin{equation}
        F_h(u_h + \tau I_H^h(u_{H}^m-u_{H}^0)) \leq F_h(u_h) + (\mu_1 + \mu_2)\gamma_h.
    \end{equation}
\end{lemma}

\begin{proof}
This follows directly from the definition of smoothed convex function given in Definition~\ref{def:smooth_convex_func}. Since there exists a value of $\tau$ satisfying the Armijo condition, we have
    \begin{align*}
        F_h(u_h + \tau I_H^h(u_{H}^m-u_{H}^0)) &\leq {}^\gamma F_h(u_h + \tau I_H^h(u_{H}^m-u_{H}^0))+\mu_1\gamma_h\\
        &\leq {}^\gamma F_h(u_h)+ c_1\tau\langle\nabla{}^\gamma
F_h(u_h),I_H^h(u_{H}^m-u_{H}^0)\rangle+\mu_1\gamma_h\\
        &\leq F_h(u_h)+(\mu_1+\mu_2) \gamma_h.
    \end{align*}
\end{proof}
The last inequality follows from the fact that, by Lemma \ref{lemma:decrease_smooth}, the term
$c_1\tau\langle\nabla{}^\gamma F_h(u_h),I_H^h(u_{H}^m-u_{H}^0)\rangle$
is negative, since $I_H^h(u_{H}^m-u_{H}^0)$ is a descent direction for ${}^\gamma F_h$.
This shows that a coarse-level minimization step leads to a decrease of $F_h$, up to the constant $(\mu_1+\mu_2)\gamma_h$, which can be made arbitrarily small by letting $\gamma_h$ tend to zero.

\subsection{Extension to a multilevel scheme}\label{sec:extension_multilevel}
The extension of the previously presented results to the multilevel framework is straightforward. Indeed, when considering $J$ levels in the multilevel approach, the analysis developed above can be applied to each pair of consecutive levels. Proceeding recursively, one can show that the coarsest level $J$ produces a bounded correction that defines a descent direction for the upper level. This correction can then be recursively exploited to generate a new descent direction for level $J-1$, and so on, until a refinement of the initial point at the finest level is obtained. In terms of classical multilevel terminology, all these computations can be interpreted as a single V--cycle of the proposed multilevel strategy.

In principle, several V--cycles could be performed consecutively. The general multilevel scheme is summarized in \Cref{fig:general_scheme_multilevel}, where we consider $J=4$ levels, perform $m$ iterations on each coarse grid, and apply $s$ fine-level iterations after each V--cycle.

We also design an adaptive strategy to select the number of coarse iterations and V--cycles automatically, based on whether the Armijo condition \eqref{eq:cond:Armijo} is satisfied, as reported in \Cref{alg:multilevel_auto,alg:Vcycle_auto}. This condition is used to estimate the parameter $\tau$, which is computed through a backtracking procedure so that \eqref{eq:cond:Armijo} holds. The idea is as follows:
\begin{enumerate}[label=(\alph*)]
    \item Perform $m$ iterations on the coarse grid and store all intermediate reconstructions.
    \item Check whether the computed reconstruction satisfies the Armijo condition.
    \item If the condition is not satisfied, select the reconstruction obtained after $\floor*{\frac{m}{2}}$ iterations and repeat the Armijo test. Continue by progressively reducing the number of iterations (e.g., $4 \rightarrow 2 \rightarrow 1$).
    \item If the reconstruction obtained after a single iteration still does not satisfy the Armijo condition, then no V--cycle is performed and no additional V--cycles are attempted.
\end{enumerate}
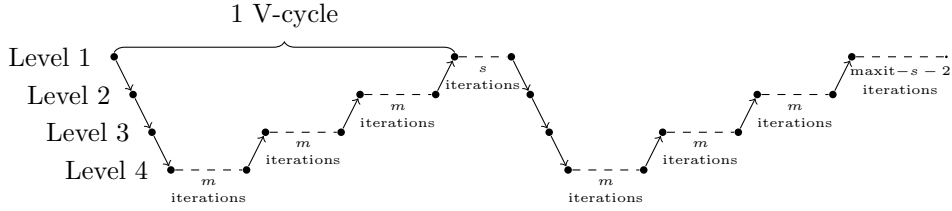
\begin{figure}[htbp]
    \centering
    \begin{tikzpicture}

\node[circle, fill=black, inner sep=1pt] (L1d1) at (0,0) {};
\node[circle, fill=black, inner sep=1pt] (L2d1) at (0.25,-0.5) {};
\node[circle, fill=black, inner sep=1pt] (L3d1) at (0.5,-1) {};
\node[circle, fill=black, inner sep=1pt] (L4d1) at (0.75,-1.5) {};
\node[circle, fill=black, inner sep=1pt] (L4u1) at (1.75,-1.5) {};
\node[circle, fill=black, inner sep=1pt] (L3u1) at (2,-1) {};
\node[circle, fill=black, inner sep=1pt] (L3u12) at (3,-1) {};
\node[circle, fill=black, inner sep=1pt] (L2u1) at (3.25,-0.5) {};
\node[circle, fill=black, inner sep=1pt] (L2u12) at (4.25,-0.5) {};
\node[circle, fill=black, inner sep=1pt] (L1u1) at (4.5,0) {};
\node[circle, fill=black, inner sep=1pt] (L1d2) at (5.25,0) {};
\node[circle, fill=black, inner sep=1pt] (L2d2) at (5.5,-0.5) {};
\node[circle, fill=black, inner sep=1pt] (L3d2) at (5.75,-1) {};
\node[circle, fill=black, inner sep=1pt] (L4d2) at (6,-1.5) {};
\node[circle, fill=black, inner sep=1pt] (L4u2) at (7,-1.5) {};
\node[circle, fill=black, inner sep=1pt] (L3u2) at (7.25,-1) {};
\node[circle, fill=black, inner sep=1pt] (L3u22) at (8.25,-1) {};
\node[circle, fill=black, inner sep=1pt] (L2u2) at (8.5,-0.5) {};
\node[circle, fill=black, inner sep=1pt] (L2u22) at (9.5,-0.5) {};
\node[circle, fill=black, inner sep=1pt] (L1u2) at (9.75,0) {};
\node[circle, fill=black, inner sep=0pt] (L1) at (11,0) {};

\draw[->] (L1d1) -- (L2d1);
\draw[->] (L2d1) -- (L3d1);
\draw[->] (L3d1) -- (L4d1);
\draw[-, dashed] (L4d1) -- (L4u1) node[midway, below, align=center]{\tiny\shortstack{$m$\\iterations}};
\draw[->] (L4u1) -- (L3u1);
\draw[-, dashed] (L3u1) -- (L3u12) node[midway, below, align=center]{\tiny\shortstack{$m$\\iterations}};
\draw[->] (L3u12) -- (L2u1);
\draw[-, dashed] (L2u1) -- (L2u12) node[midway, below, align=center]{\tiny\shortstack{$m$\\iterations}};
\draw[->] (L2u12) -- (L1u1);
\draw[-,dashed] (L1u1) -- (L1d2)node[midway, below, align=center]{\tiny\shortstack{$s$\\iterations}};
\draw[->] (L1d2) -- (L2d2);
\draw[->] (L2d2) -- (L3d2);
\draw[->] (L3d2) -- (L4d2);
\draw[-, dashed] (L4d2) -- (L4u2)node[midway, below, align=center]{\tiny\shortstack{$m$\\iterations}};
\draw[->] (L4u2) -- (L3u2);
\draw[-, dashed] (L3u2) -- (L3u22)node[midway, below, align=center]{\tiny\shortstack{$m$\\iterations}};
\draw[->] (L3u22) -- (L2u2);
\draw[-, dashed] (L2u2) -- (L2u22)node[midway, below, align=center]{\tiny\shortstack{$m$\\iterations}};
\draw[->] (L2u22) -- (L1u2);
\draw[-, dashed] (L1u2) -- (L1)node[midway, below, align=center]{\tiny\shortstack{maxit$-s-2$\\iterations}};

\draw[decorate, decoration={brace, amplitude=4pt}] (L1d1.north) -- (L1u1.north) node[midway, above=6pt] {1 V-cycle};

\node[left=5pt] at (L1d1) {Level 1};
\node[left=5pt] at (L2d1) {Level 2};
\node[left=5pt] at (L3d1) {Level 3};
\node[left=5pt] at (L4d1) {Level 4};

\end{tikzpicture}
    \caption{General scheme of the multilevel approach with $4$ levels, $m$ coarse iterations and $\text{maxit}$ fine iterations.}
    \label{fig:general_scheme_multilevel}
\end{figure}

\begin{algorithm}[htbp]
\caption{MULTILEVEL\_AUTO\quad (2 levels)}\label{alg:multilevel_auto}
\begin{algorithmic}[1]
\State Set $u_{h}^0 \in \mathbb{R}^N$, $p$, $\mathrm{max_{bt}}, m, \Phi, \mathrm{cycle} = \text{true}$.
\While { $\mathrm{cycle} == \text{true}$}
\State [$u_{h}^{k},$ $\mathrm{cycle}]=
$\,\textbf{Vcycle\_iter\_auto($u_h^{k-1},m,\Phi_H,\mathrm{max_{bt}}$)}
\EndWhile
\State $u_{h}^{\mathrm{maxit}} = \Phi_{h}(u_h^p, \mathrm{maxit}-p)$
\end{algorithmic}
\end{algorithm}

\begin{algorithm}[htbp]
\caption{Vcycle\_iter\_auto\quad (2 levels)}\label{alg:Vcycle_auto}
\begin{algorithmic}[1]
\State $u_H^{k,0} = I_h^H u_h^k$
\State $v_{H}^k = I_h^H \, \nabla F_{h}(u_h^k) - \nabla ^\gamma F_{H}(u_H^{k,0})$
\State $u_H^{k,m} = \Phi_H(u_H^{k,0}, m, v_H^k)$
\State $\mathrm{success} = \text{false}$
\While{$\mathrm{success} = \text{false}$ \textbf{and} $m \geq 1$}
\State $c_h^{k,m}= I_H^h\left(u_H^{k,m}-u_H^{k,0}\right)$
\State $\tau^{k}=5$ ,\qquad $t=0$
\While{$F_h(u_h^k+\tau^{k}c_h^{k,m})> F_h(u_h^k) + c\tau^{k}\nabla F_h^T(u_h^k)c_h^{k,m}\land t < \mathrm{max_{bt}}$}
\State $\tau^{k} = (0.85)^t\,\tau^{k}$,\quad $t = t+1$
\EndWhile
\If{$t = \mathrm{max_{bt}}$}
\State $m=\floor*{\frac{m}{2}}$
\Else
\State $\mathrm{success} = \text{true}$
\EndIf
\EndWhile
\If { $\mathrm{success} = \text{true}$}
\State $u_{h}^{k+1} = u_h^k +\tau^{k}c_h^{k,m}$
\Else
\State $\mathrm{cycle} = \text{false}$
\EndIf
\end{algorithmic}
\end{algorithm}

\section{Convergence}\label{sec:conv}
In this section, we show that the considered multilevel framework generates a convergent sequence of iterates $\{u^k\}$ converging to a solution of the original problem \eqref{eq:problem_leastsquares}. For a more detailed discussion of the convergence analysis, we refer the reader to \cite{IML_FISTA}. In order to study the convergence of the multilevel approach, it is necessary to understand the type of approximation introduced when computing the proximity operator of the regularization term. To this end, we first extend the notion of subdifferential through the following definition:
\begin{definition}
    The $\epsilon$-subdifferential of $R$ at $u\in\dom(R)$ is defined as:
    \begin{equation*}
        \partial_{\epsilon}R(u)=\{y\in\mathbb{R}^d \mid R(u)\ge R(u) +\langle x-u,y\rangle-\epsilon, \forall x\in\mathbb{R}^d\}.
    \end{equation*}
\end{definition}

Based on this definition, different types of approximations can be introduced. However, typical image restoration problems in which dual optimization is employed are based on what is called type $2$ approximation. 

\begin{definition}\label{def:type_2_approx}
    We say that $x\in\mathbb{R}^d$ is a type $2$ approximation of $\prox_{\gamma R}(u)$ with precision $\epsilon$, and we write $x\approx_{2,\epsilon}\prox_{\gamma R}(u)$, if and only if
    \begin{equation*}
        \gamma^{-1}(u-x)\in\partial_{\epsilon}R(x).
    \end{equation*}
\end{definition}

The convergence results corresponding to all the approximation types, as well as the case in which the proximal operator can be computed in closed form, are discussed in detail in \cite{IML_FISTA}. 
In what follows, we show that the strategy adopted here to approximate the proximity operator coincides with the approach proposed in \cite{IML_FISTA}, and therefore yields a type $2$ approximation.

Firstly, recall that in our case the regularization term is given by $R(u) = (g\circ W)(u)$. The NPD and PNPD methods, described in Section~\ref{sec:inex_prox_methods}, compute an approximation of the proximal operator based on Lemma \ref{lemma:primal_dual_NPD} that is they compute a dual sequence $v^{k}$ defined as
\begin{equation}\label{eq:dual_sequence}
    v^{k+1} = \prox_{\beta \alpha^{-1}g^*}(v^k+\beta\alpha^{-1}W(a-\alpha W^Tv^k)), \quad \forall k \geq 0,
\end{equation}
where $\alpha>0$, $0<\beta<2/\|W\|^2$ and $v^0\in\mathbb{R}^{d'}$.
To show that this coincides with a type 2 approximation, we will show that this coincides with applying any proximal gradient methods to a particular minimization problem. Indeed, starting from the standard definition of the proximity operator, we switch to the dual formulation by using the convex conjugate defined in \Cref{def:convex_conj}. Namely,
\begin{align}
    \prox_{\gamma g\circ W}(u) 
    &= \argmin_{x\in\mathbb{R}^d}\frac{1}{2}\|x-u\|_2^2+\gamma(g(Wx)) \notag\\
    &= \argmin_{x\in\mathbb{R}^d}\max_{w\in\mathbb{R}^d}\frac{1}{2}\|x-u\|_2^2 + \gamma(x^TW^Tw-g^*(w)) \notag\\
    &= \argmax_{w\in\mathbb{R}^d}-\frac{1}{2\gamma}\|u-\gamma W^Tw\|_2^2+\frac{1}{2\gamma}\|u\|_2^2-g^*(w) \notag\\
    &= \argmin_{w\in\mathbb{R}^d}\frac{1}{2\gamma}\|u-\gamma W^Tw\|_2^2+g^*(w).
    \label{eq:dual_prox}
\end{align}
Problem \eqref{eq:dual_prox} is usually referred to as the dual problem. An approximate solution can be computed by applying any suitable optimization method. For instance, by using a proximal gradient method, one obtains the iterative scheme
\begin{equation*}
    w^{k+1} = \prox_{\mu g^*}(w^{k}+\mu W(u-\gamma W^Tw^{k})),
\end{equation*}
where the parameter $\mu$ must be chosen according to the Lipschitz constant of the gradient of the differentiable term $\frac{1}{2\gamma}\|u-\gamma W^Tw\|_2^2$. In particular, one requires $\mu < \frac{2}{\gamma \|WW^T\|}$. Therefore, by writing $\mu = \beta\gamma^{-1}$ with $0<\beta < 2/\|W\|_2^2$, we recover exactly the same iterative scheme as in \eqref{eq:dual_sequence}. Dual optimization thus provides a simple way to approximate the proximity operator, while also offering theoretical guarantees on the computed approximation, as stated in the following proposition.

\begin{proposition}
    Assume that $(w_k)_{k\in\mathbb{N}}$ is a minimizing sequence for the dual problem \eqref{eq:dual_prox}. Then:
    \begin{itemize}
        \item The sequence $(u-\gamma W^Tw_k)_{k\in\mathbb{N}}$ converges to the exact solution of $\prox_{\gamma g\circ W}(u)$.
        \item This sequence provides a type $2$ approximation of the proximity operator.
    \end{itemize}
\end{proposition}
\begin{proof}
    See \cite{IML_FISTA}, Proposition 3.5.
\end{proof}

The convergence guarantees of the overall algorithm are directly related to the accuracy with which the proximity operator is approximated at each iteration. Typically, a smaller approximation error requires a greater computational effort, which is why some approaches use a fixed budget of inner iterations to estimate the proximity operator. This choice, however, comes at the price of limited accuracy and may eventually lead to divergence when a large number of outer iterations are performed.

In \cite{IML_FISTA}, the authors propose a heuristic strategy based on a tolerance parameter controlling the relative distance between two consecutive inner iterates used in the approximation of the proximity operator. Whenever no decrease of the objective function is observed over two consecutive iterations, the tolerance is reduced, thereby enforcing a more accurate computation of the proximal point.

In our numerical experiments, we choose to keep the number of inner iterations fixed when computing the proximity operator. More precisely, we perform only one inner iteration. This choice is motivated by the results in \cite{PNPD,NPD}, where the NPD and PNPD methods do not exhibit any substantial improvement when the number of nested iterations is increased. We also empirically verified this behavior by testing a larger number of inner iterations. Although, in principle, this should yield a more accurate approximation of the proximity operator, in all our experiments, the results remain essentially unchanged compared with the case in which only one inner iteration is performed.

\section{Experimental setup}\label{sec:exp_set}
In this section, we describe the setting and the implementation details of the multilevel framework used in our numerical experiments.

\paragraph{Variational model}
In the numerical experiments, we consider image problems 
of the form \eqref{eq:Model_Equation} solving the optimization problem \eqref{eq:problem_leastsquares}.

We consider two different regularization terms. In the first case, we take $g(Wu) = \|Wu\|_1$, where $W\in\mathbb{R}^{3d\times d}$ is the framelet operator defined in Appendix~\ref{app:framelets}. In this case, the optimization problem \eqref{eq:problem_leastsquares} takes the form
\begin{equation}\label{eq:exact_case_model}
    \underset{u\in\mathbb{R}^d}{\operatorname{argmin}} \ \frac{1}{2} \|Au-b^{\delta}\|_2^2+\lambda\|Wu\|_1.
\end{equation}
This type of variational problem can be efficiently addressed by exact proximal gradient methods such as FISTA and ITTA, since the proximity operator for the regularization term can be computed in closed form. For this reason, in the numerical experiments, we analyze the performance of the multilevel framework when FISTA or ITTA are used as solvers at both the fine and the coarse levels.

The second regularization term we consider is the total variation (TV) functional \cite{Rudin-Osher-Fatemi-1992}, that is
\begin{equation}\label{eq:inexact_case_model}
     \underset{u\in\mathbb{R}^d}{\operatorname{argmin}} \ \frac{1}{2}\|Au-b^{\delta}\|_2^2+\lambda TV(u).
\end{equation}
In this case, the proximity operator of the TV does not admit a closed-form expression, and therefore only approximate solutions of the optimization problem \eqref{eq:inexact_case_model} can be computed by means of inexact proximal gradient methods such as NPD and PNPD, see \Cref{sec:inex_prox_methods}. Also in this setting, we analyze the performance achieved by the multilevel approach when inexact proximal gradient methods are used as solvers at both the fine and coarse levels, and compare it with the corresponding plain strategy.

To this end, we conclude this preliminary part by recalling that the TV functional can be written in the form $g\circ W$, where
\begin{itemize}
  \item \(W : \mathbb{R}^d \to \mathbb{R}^{2d}\) is the discrete gradient operator,
        \[
            Wu = 
            \begin{bmatrix}
                D_x u \\
                D_y u
            \end{bmatrix},
        \]
        with \(D_x\) and \(D_y\) denoting the finite-difference operators in the horizontal
        and vertical directions, respectively.
  \item \(h : \mathbb{R}^{2d} \to \mathbb{R}\) is the convex function
        \[
            h(z) = \sum_{j=1}^{d} \lVert z_j \rVert_2,
        \]
        where \(z_j \in \mathbb{R}^2\) is the two-dimensional gradient vector at pixel \(j\).
\end{itemize}

\begin{remark}\label{rem:beta}
    The choice of $\beta$ in the NPD and PNPD algorithms must be consistent with \Cref{alg:npd,alg:pnpd}. In the TV case, it can be shown that $\beta = \frac{1}{8}$ satisfies the requirement $0<\beta<\frac{2}{\|W\|^2}$.
\end{remark}

\paragraph{Transfer operators}

As mentioned in \Cref{sec:coarse_model_trans_operators}, a crucial ingredient of multilevel methods is the ability to transfer information between different levels. From a theoretical point of view, the only assumption required is that the transfer operators be coherent according to Definition~\ref{def:CIT}. In this section, we describe a classic choice of CIT transfer operators based on linear B-splines.

The restriction operator is the full-weighting associated with the stencil
\[\mathcal{R} = \frac{1}{16}
\begin{bmatrix}
1 & 2 & 1 \\
2 & 4 & 2 \\
1 & 2 & 1
\end{bmatrix},
\]
which is centered on the target pixel and computes a weighted average of the neighboring pixels.

To define the coarser problem, we combine the linear operator associated with the stencil $\mathcal{R}$ with the downsampling operator. Consequently, the resulting restriction operator can be written as
\begin{align*}
I_h^H& = \frac{1}{16}
\begin{pmatrix}
2 & 1 & 0 & \cdots &  && 0\\
0 & 1 & 2 & 1 & 0 & \cdots & 0 \\
\vdots & & \ddots & \ddots & \ddots & & \vdots \\
0 & \cdots & & 0 & 1 & 2& 1 \\
\end{pmatrix}
\otimes
\begin{pmatrix}
2 & 1 & 0 & \cdots &  && 0\\
0 & 1 & 2 & 1 & 0 & \cdots &0 \\
\vdots & & \ddots & \ddots & \ddots & & \vdots \\
0 & \cdots & & 0 & 1 & 2& 1 \\
\end{pmatrix},
\end{align*}
where $\otimes$ denotes the Kronecker product.
Note the similarity with the matrix $W_{00}$ in Appendix~\ref{app:framelets}.

Correspondingly, the stencil associated with the prolongation operator is defined as
\[
\mathcal{P} = \frac{1}{4}
\begin{bmatrix}
1 & 2 & 1 \\
2 & 4 & 2 \\
1 & 2 & 1
\end{bmatrix}
= 4 \mathcal{R}^T,
\]
which defines the bilinear interpolation operator
\begin{align*}
I_H^h &= 4(I_h^H)^T.
\end{align*}
The pair $(I_h^H, I_H^h)$ provides a simple and intuitive construction of CIT operators.

\paragraph{Coarse problem}
In Section \ref{sec:coarse_model_trans_operators}, we described how to properly define the variational problem at coarser levels starting from the fine-level formulation. When dealing with image deblurring problems, a further crucial aspect is the definition of the blurring operator $A$ on the coarser grids. To this end, we briefly describe below the strategy adopted in our numerical experiments.

We consider the two-level setting and let 
\[d = d_h > d_H=\frac{d}{4},\]
so that $I_h^H \in \mathbb{R}^{d_H \times d_h}$ and $I_H^h \in \mathbb{R}^{d_h \times d_H}$.
Then, we denote by $A_h=A$ the blurring operator at the fine level, and we construct the coarse-level operator according to the Galerkin approach, as
\begin{equation}\label{eq:A_H}
A_H=I_h^H A_hI_H^h.
\end{equation}
The interpretation of this formulation for $A_H$ is that applying $A_H$ to a vector is equivalent to interpolating the vector to the finer level, applying the operator $A_h$, and then restricting the result to the coarser level.

\paragraph{Computational cost}
In the numerical experiments, we compare the standard methods FISTA, ISTA, NPD, and PNPD with their multilevel counterparts. To ensure a fair comparison, we explicitly analyze the computational cost in all cases.

Since the coarse problem has size $d/4$, estimates suggest that four iterations at the coarse level are approximately equivalent, in terms of floating-point operations, to one iteration at the fine level. 
For this reason, the number of coarse iterations in the multilevel methods is chosen to be a multiple of four. In particular, in the plots showing the error versus the iteration count, one multilevel iteration with four coarse iterations is compared with one fine-level iteration, whereas one multilevel iteration with eight coarse iterations is compared with two fine-level iterations. For a comprehensive analysis of the computational cost, we refer the reader to Appendix \ref{app:computational_cost}.

\section{Numerical results}\label{sec:num_res} 
All numerical tests were carried out in MATLAB 25.2.0.2998904 (R2025b) on a laptop running Windows~11 Home (version~25H2), equipped with an Intel Core i7-1065G7 processor (1.30~GHz) and 16~GB of RAM. The test problems were generated using the IRtools toolbox \cite{ITtools}.

For all algorithms, the initial guess is chosen as the observed image. The shifting parameter of the preconditioner is fixed at
$\nu = 0.1$ throughout all numerical experiments. In the multilevel framework, the number of levels is set to $2$, the Moreau parameter is chosen as $\gamma = 1.1$, and a total of $100$ iterations are performed. 

Since the coarse-level problem is better conditioned, the regularization parameter and the preconditioner shift at the coarse level are set to
$\lambda_H=\lambda_h/4$  and $\nu_H=\nu_h/4$,
respectively.
Moreover, periodic boundary conditions are assumed in the construction of the matrix $A$, so that all computations can be performed by means of the Fast Fourier Transform (FFT).

We tested several values of $m$, corresponding to the number of coarse-grid iterations, and of $p$, corresponding to the number of V-cycles. Among the various numerical experiments performed, the choice $m = 8$ and $p = 8$ provided a good compromise across all examples, independently of the blur and noise levels. Therefore, in the multilevel framework (Algorithm~\ref{alg:multilevel}), we perform $8$ iterations on the coarse grid for $8$ V-cycles. No intermediate fine-level iterations after a V-cycle iteration are performed, since our experiments indicate that they slow down convergence.

Our main goal is to show that the multilevel and preconditioned algorithms are more efficient than their non-multilevel and non-preconditioned counterparts for solving deblurring problems, without compromising the quality of the computed reconstructions, which remains comparable across all algorithms.
The performance of the methods is assessed using two metrics.
The first one is the \textit{Relative Reconstruction Error} (RRE), defined as
\begin{equation}
    \label{eq:RRE}
    \text{RRE}(u_{\text{rec}}) = \frac{\|u_{\text{true}} - u_{\text{rec}}\|}{\|u_{\text{true}}\|},
\end{equation}
where \(u_{\text{rec}}\) denotes the reconstructed image and \(u_{\text{true}}\) the ground-truth image.
The second metric is the \textit{Peak Signal-to-Noise Ratio} (PSNR), defined as 
\begin{equation}
    \label{eq:PSNR}
    \text{PSNR}(u_{\text{rec}})= 10\log_{10} \left(\frac{1}{\|u_{\text{true}} - u_{\text{rec}}\|^2}\right),
\end{equation}
since our gray-scaled images are scaled in $[0,1]$. Higher PSNR values correspond to more accurate reconstructions.

\subsection{Exact case}\label{subsec:ex_num_res} 
This section focuses on the solution of the minimization problem \eqref{eq:exact_case_model}, where the proximity operator can be computed in closed form. In the numerical experiments, we compare the standard proximal methods ITTA and FISTA with their multilevel counterparts, namely MITTA and MM\_FISTA. In addition, we test the automatic strategy presented in \Cref{sec:extension_multilevel} for estimating the optimal number of coarse iterations and V-cycles. The multilevel method based on this adaptive strategy is called MITTA\_auto.

To ensure convergence of the methods, at each level we set the step-length parameter $\alpha = 1$ for ITTA, while for FISTA we adopt the classic choice $\alpha = 1/L$, where $L$ denotes the Lipschitz constant of the gradient of the data-fidelity term.

To test the different methods, we consider a grayscale satellite image of $512~\times~512$ pixels, blurred by a Gaussian PSF. The final observed image $b^{\delta}$ is obtained by adding $2\%$ white Gaussian noise $\eta_\delta$, that is,
\[
\|\eta_\delta\| = 0.02\|b\|.
\]
\subsubsection{Example 1}
In this first example, the PSF is of $35\times 35$ pixels. Figure~\ref{fig:sat_GaussMedium} shows the ground-truth image, the PSF, and the corresponding observed image. 

\begin{figure}
    \centering
    \makebox[\textwidth][c]{
        \begin{subfigure}{0.25\textwidth}
            \centering
            \includegraphics[width=\textwidth]{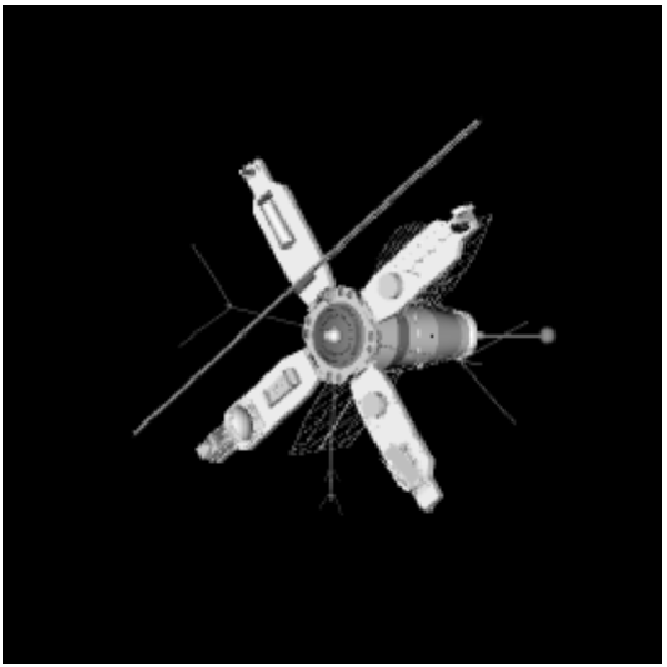}
            \caption{}\label{fig:sat_groundtruth}
        \end{subfigure}
        \begin{subfigure}{0.25\textwidth}
            \centering
            \includegraphics[width=\textwidth]{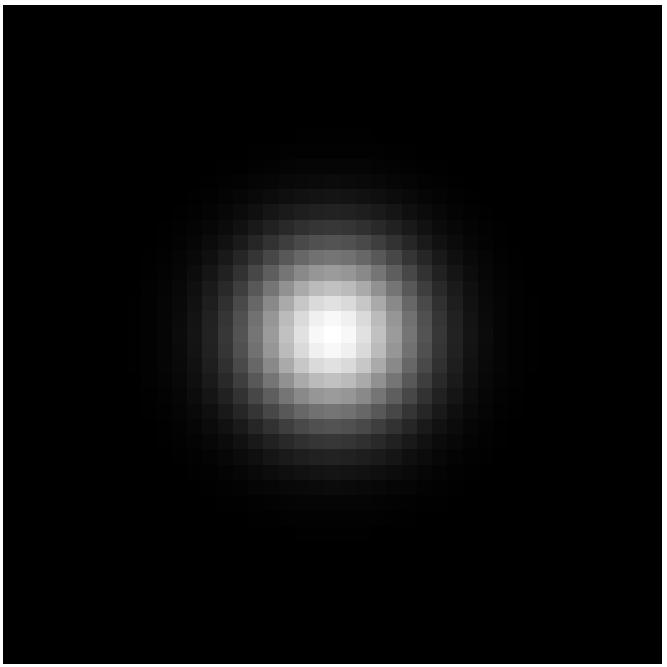}
            \caption{}\label{fig:psf_GaussMedium}
        \end{subfigure}
        \begin{subfigure}{0.25\textwidth}
            \centering
            \includegraphics[width=\textwidth]{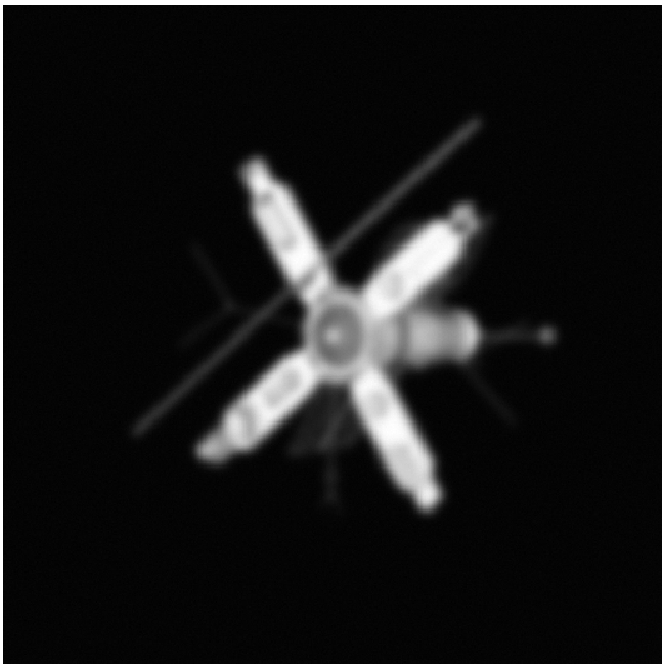}
            \caption{}\label{fig:b_GaussMedium}
        \end{subfigure}
    }
    \caption{Example 1: (\subref{fig:sat_groundtruth}) True image, (\subref{fig:psf_GaussMedium}) Gaussian PSF of $35\times 35$ pixels, (\subref{fig:b_GaussMedium}) Observed image.}
    \label{fig:sat_GaussMedium}
\end{figure}

The regularization parameter is $\lambda = 3 \cdot 10^{-4}$ for FISTA and MM\_FISTA, while $\lambda = 2 \cdot 10^{-3}$ for ITTA, MITTA, and MITTA\_auto.
MITTA\_auto performs $8$ coarse iterations and $3$ V-cycles. 

\begin{figure}
     \centering
     \begin{overpic}[width=.6\linewidth]{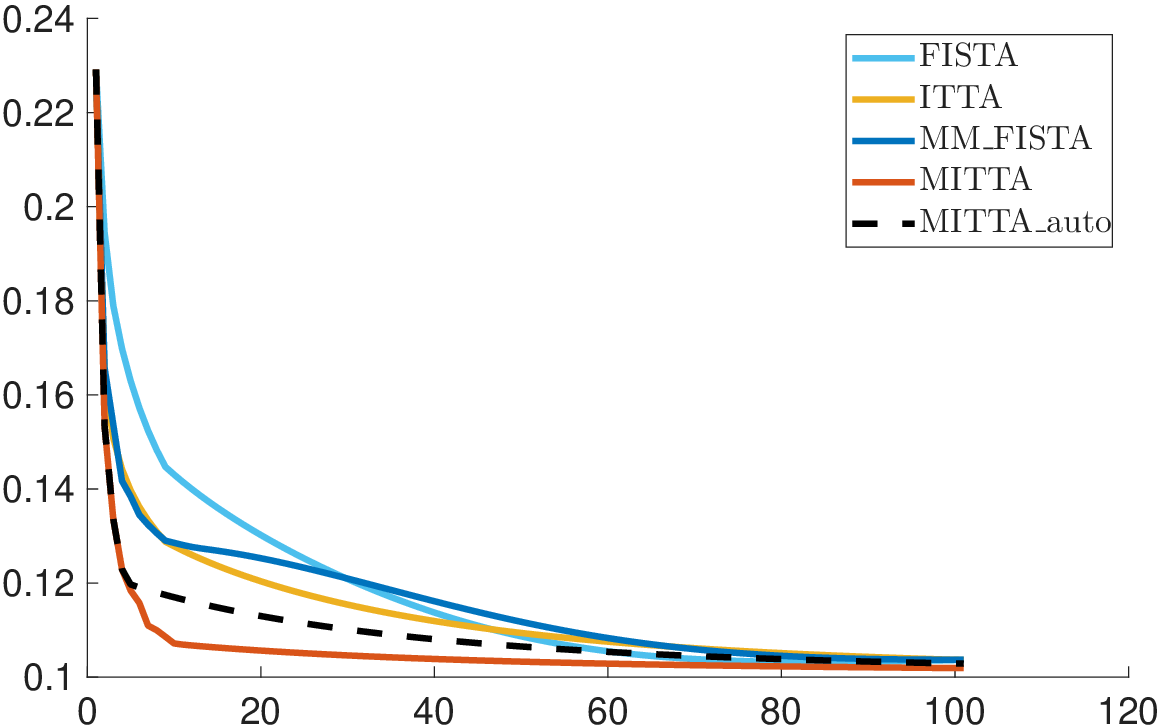}
        \put(32,-6){Number of iterations} 
        \put(-7,28){\rotatebox{90}{RRE}} 
    \end{overpic}
    \vspace{8mm}
     \caption{Example 1: comparison of the RREs varying the number of iterations for FISTA, ITTA, and their multilevel versions.}
     \label{fig:err_sat_GaussMedium}
\end{figure}

Figure~\ref{fig:err_sat_GaussMedium} shows the RRE values obtained by the considered methods, varying the number of iterations, and plotted according to the previous cost analysis. We observe that, in terms of iteration count, both MITTA and MITTA\_auto converge faster than ITTA. In particular, the initial coarse-grid iterations already provide a significant improvement of the approximation. Moreover, thanks to the presence of the preconditioner $P$, both ITTA and its multilevel counterpart show a faster decrease of the RRE than FISTA and MM\_FISTA.
Although the automatic multilevel procedure is slightly slower than the multilevel version in which the number of coarse iterations and V-cycles is manually optimized, it still outperforms all the other methods. Since it avoids the need to tune two additional parameters, it provides a competitive and robust alternative in practice.

To better emphasize the acceleration of the convergence in the first iterations, Table~\ref{tab:ex_1} shows the PSNR for the considered methods after $4$, $15$, and $60$ iterations.
Note that, for instance, the reconstruction produced by MITTA after only $4$ V-cycles attains a higher PSNR than the ITTA reconstruction after $15$ iterations, confirming the improved performance of the multilevel approach.
This is further highlighted by the detail of the reconstructions shown in Figure~\ref{fig:dett_satMedium_ITTA}.
 
 \begin{table}
\centering
\begin{tabular}{>{\centering\arraybackslash}m{0.01\textwidth}
    >{\centering\arraybackslash}m{0.13\textwidth}
    >{\centering\arraybackslash}m{0.15\textwidth}
    >{\centering\arraybackslash}m{0.16\textwidth}
    >{\centering\arraybackslash}m{0.16\textwidth}
    >{\centering\arraybackslash}m{0.16\textwidth}} 
 Iter & \textcolor{fista}{FISTA} & \textcolor{itta}{ITTA} & \textcolor{mfista}{MM\_FISTA} & \textcolor{mitta}{MITTA} & \textcolor{black}{MITTA\_auto}\\[0.5ex] 
 \hline\\[-0.25cm]
 4 & 29.3752 & 30.7140 & 30.7939 & 32.1489 & 32.1001\\ 
 15 & 31.0242& 31.8205& 31.5673&33.0966&32.5803\\
 60 & 33.1690&32.9968&32.9448&33.3729&33.2069 \\ 
 \hline \\
\end{tabular}
\caption{Example 1: PSNR values obtained at different iterations with different methods.}
\label{tab:ex_1}
\end{table}

\begin{figure}
    \centering
    \begin{tabular}{ccc}
        \includegraphics[width=0.25\linewidth]{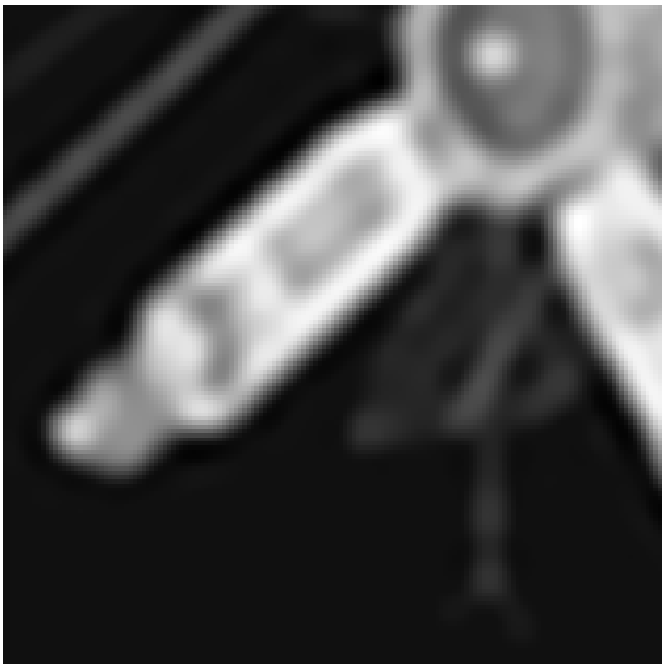} &
        \includegraphics[width=0.25\linewidth]{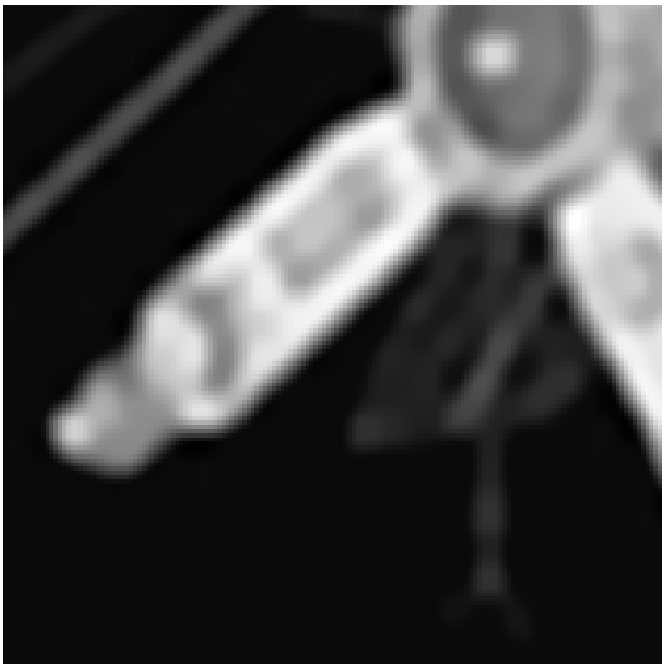} &
        \includegraphics[width=0.25\linewidth]{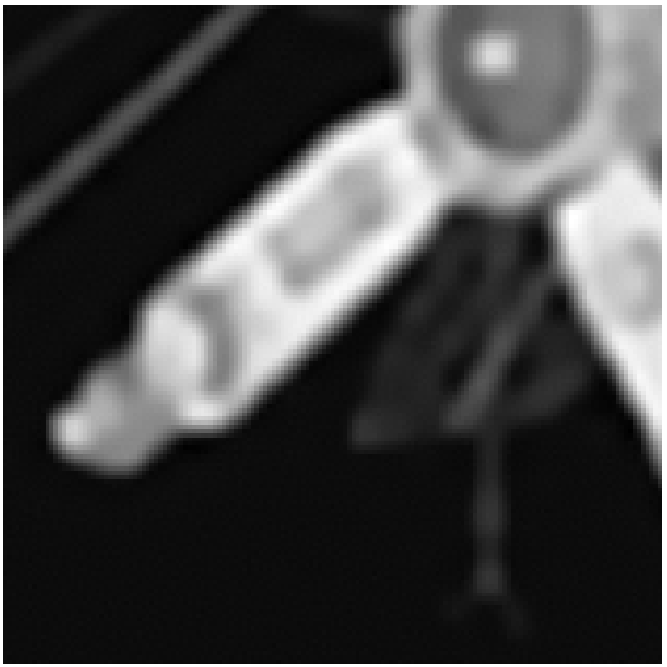}\\
        \text{\small 24.9971} & \text{\small 25.9777} & \text{\small 26.3410}
    \end{tabular}
    \caption{Example 1: detailed view of the reconstructed images. From left to right, we show the ITTA reconstruction after 4 iterations, after 15 iterations, and the MITTA reconstruction after 4 iterations. The corresponding PSNR values are reported below each image.}
    \label{fig:dett_satMedium_ITTA}
\end{figure}

\subsubsection{Example 2}
In this second example, we choose a more severe Gaussian PSF, namely, we consider a PSF of size $51\times 51$. Figure~\ref{fig:sat_GaussSevere} shows the ground-truth image, the corresponding PSF, and the corrupted observation used in this experiment.
\begin{figure}
    \centering
    \makebox[\textwidth][c]{
        \begin{subfigure}{0.25\textwidth}
            \centering
            \includegraphics[width=\textwidth]{Images/Satellite.eps}
            \caption{}
        \end{subfigure}
        \begin{subfigure}{0.25\textwidth}
            \centering
            \includegraphics[width=\textwidth]{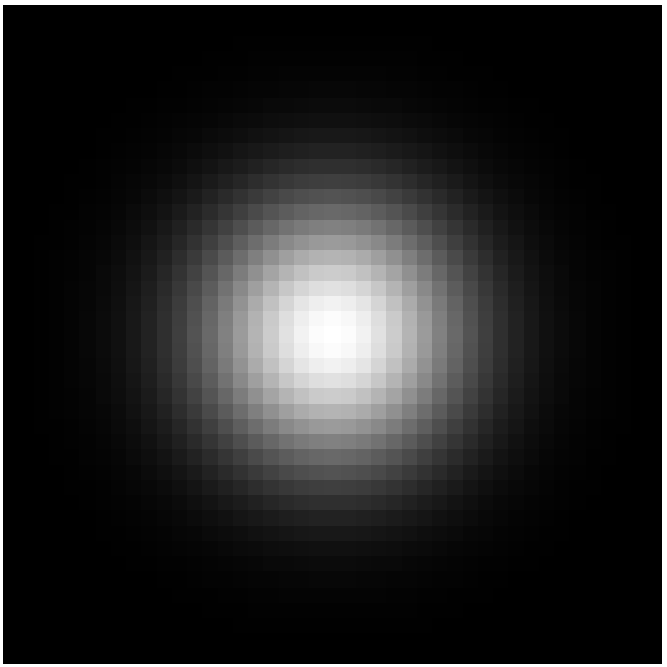}
            \caption{}\label{fig:psf_GaussSevere}
        \end{subfigure}
        \begin{subfigure}{0.25\textwidth}
            \centering
            \includegraphics[width=\textwidth]{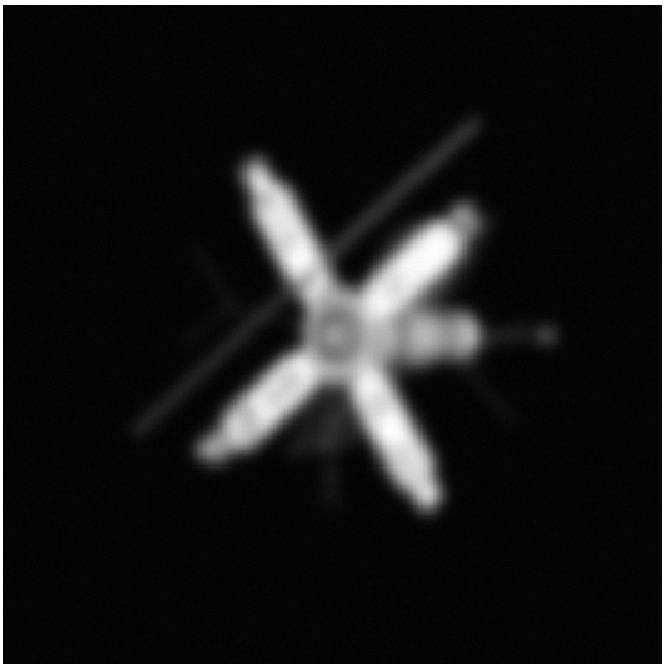}
            \caption{}\label{fig:b_GaussSevere}
        \end{subfigure}
    }
    \caption{Example 2:(\subref{fig:sat_groundtruth}) True image, (\subref{fig:psf_GaussSevere}) Gaussian PSF of dimension $51\times 51$, (\subref{fig:b_GaussSevere}) Observed image.}
    \label{fig:sat_GaussSevere}
\end{figure}

The regularization parameter is $\lambda = 5 \cdot 10^{-4}$ for FISTA and MM\_FISTA, while it is $\lambda = 3 \cdot 10^{-3}$ for ITTA, MITTA and MITTA\_auto.
The automatic strategy selects the same number of coarse iterations and V-cycles as in the previous example, namely $8$ coarse iterations and $3$ V-cycles.

\begin{figure}
     \centering
     \begin{overpic}[width=.6\linewidth]{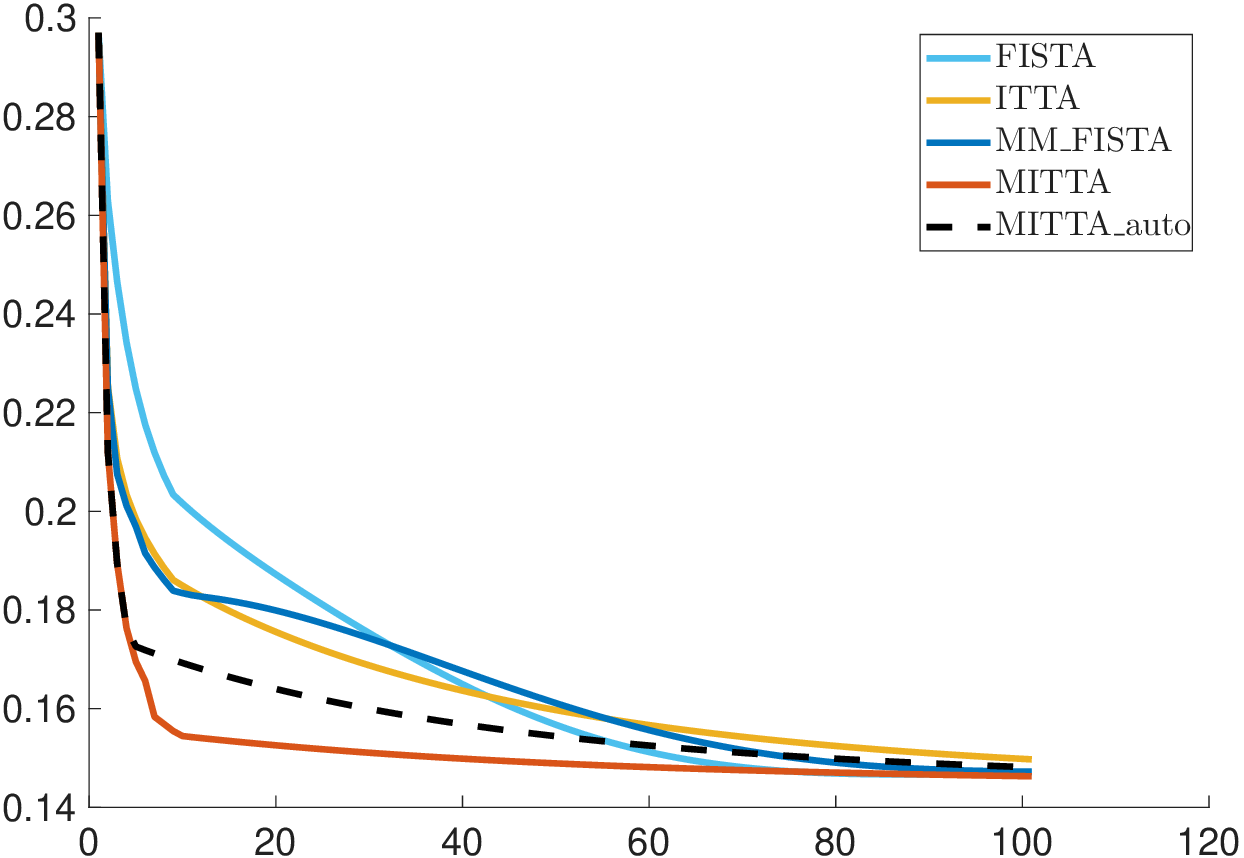}
        \put(32,-6){Number of iterations} 
        \put(-7,28){\rotatebox{90}{RRE}} 
    \end{overpic}
    \vspace{8mm}
     \caption{Example 2: comparison of the RREs varying the number of iterations for FISTA, ITTA, and their multilevel versions.}
     \label{fig:err_sat_GaussSevere}
\end{figure}

Figure~\ref{fig:err_sat_GaussSevere} reports the comparison in terms of RRE at each iteration. The plot clearly shows that the iterations performed on the coarse grid significantly accelerate convergence in the multilevel schemes.
As in the previous example, MITTA achieves the highest PSNR values among all the methods under comparison. 

Comparing our MITTA with MM\_FISTA, we note that the improvement in the reconstructed images is visible after just a few iterations. In fact, as highlighted in the zoomed-in view in Figure~\ref{fig:dett_satSevere_multilevel}, the satellite image appears sharper, particularly along the edges, when using MITTA rather than MM\_FISTA.

\begin{figure}
    \centering
    \includegraphics[width=0.25\linewidth]{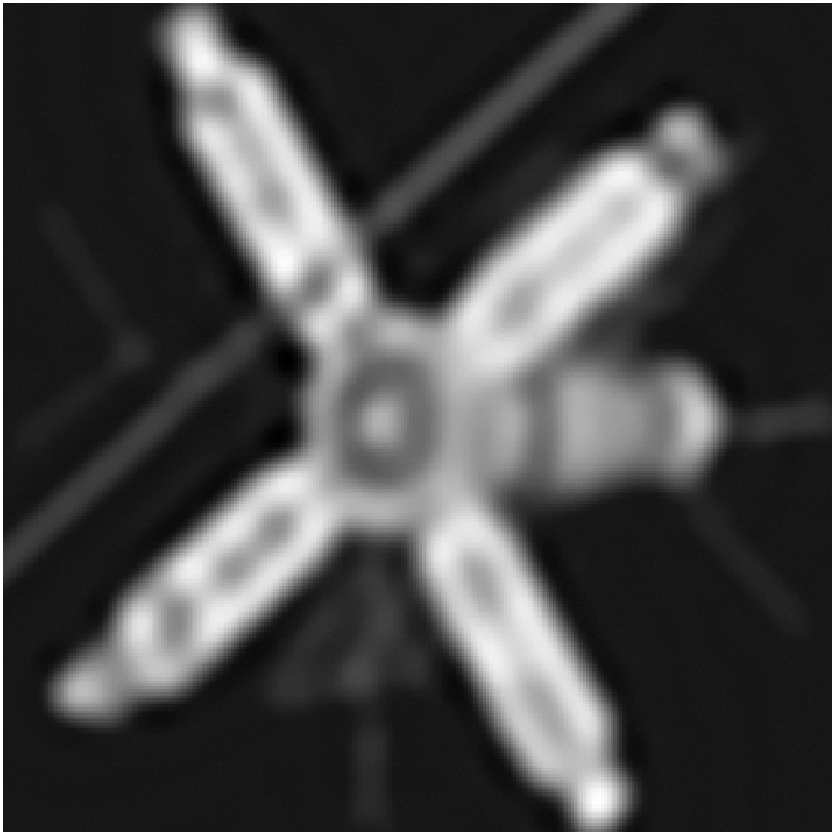}\qquad
    \includegraphics[width=0.25\linewidth]{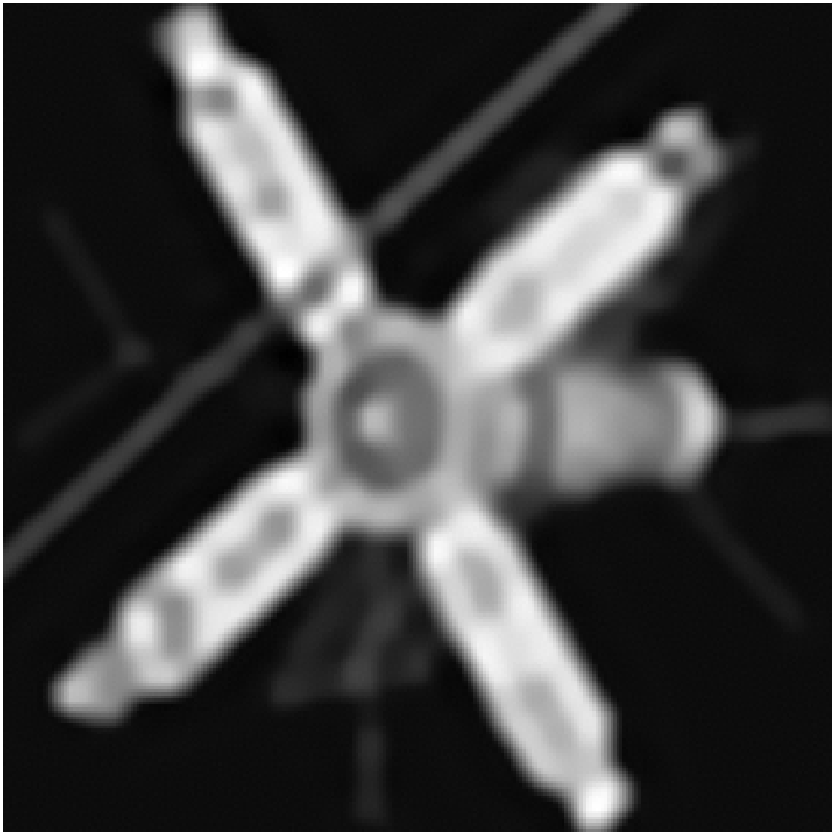}\\[0.1cm]
    \caption{Example 2: detailed view of the reconstructed images. MM\_FISTA reconstruction after 4 iterations (left) and MITTA reconstruction after 4 iterations (right).}
    \label{fig:dett_satSevere_multilevel}
\end{figure}

\subsection{Inexact case}\label{subsec:inex_num_res}
The present section focuses on the solution of the optimization problem \eqref{eq:inexact_case_model}, where the proximity operator is not available in closed form. In the numerical experiments, we compare the standard inexact proximal methods NPD and PNPD with their multilevel counterparts, namely MNPD and MPNPD. In addition, we also test the automatic strategy introduced in \Cref{sec:extension_multilevel} for estimating the optimal number of coarse iterations and V-cycles in the MPNPD method. The multilevel method based on this adaptive strategy is called MPNPD\_auto.

To guarantee convergence at both levels, we set the step length in the gradient-descent step equal to $\alpha = 1$ for PNPD and $\alpha = \frac{1}{L}$ for NPD, while for both algorithms, we choose $\beta = \frac{1}{8}$ according to the estimation in Remark \ref{rem:beta}.\\
In the NPD method, for the extrapolation sequence $\{\gamma_k\}$, we adopt the same strategy proposed in \cite{NPD}, denoted by $\{\gamma_k^{NPD}\}$ and defined as
\[\begin{cases}
    \gamma_0^{NPD} = 0\\
    \gamma_k^{NPD} = \min\left\{\gamma_k^{FISTA},\frac{C\rho_k}{\|u_k- u_{k-1}\|}\right\}, \qquad\forall k \geq \ 1,
\end{cases}\]
where $C > 0$ is a positive constant, set here to $C=100$, and $\{\rho_k\}_{k\in\mathbb{N}}$ is a fixed positive summable sequence chosen as
$\rho_k = \frac{1}{k^{1.1}}$,
as suggested in \cite{NPD}. The sequence $\gamma_k^{FISTA}$ corresponds to the classical FISTA extrapolation parameter, namely
\[\gamma_k^{FISTA} = \frac{t_k-1}{t_{k+1}}, \qquad
\mbox{where} \qquad 
\begin{cases}
    t_0 = 0,\\
    t_{k+1} = \frac{1+\sqrt{1+4t_k^2}}{2}, \qquad \forall k\geq 1.
\end{cases}\]

In the following experiments, we consider a $256 \times 256$ grayscale image of peppers, blurred with two different PSFs generated by the IRtools toolbox \cite{ITtools}, and contaminated with white Gaussian noise with intensity level equal to $1\%$.

\subsubsection{Example 3}
In this example, we consider a defocus PSF of size $31\times 31$. Figure~\ref{fig:pep_DefocusMedium} depicts the ground-truth image, the corresponding PSF, and the corrupted observation.

\begin{figure}
    \centering
    \makebox[\textwidth][c]{
        \begin{subfigure}{0.25\textwidth}
            \centering
            \includegraphics[width=\textwidth]{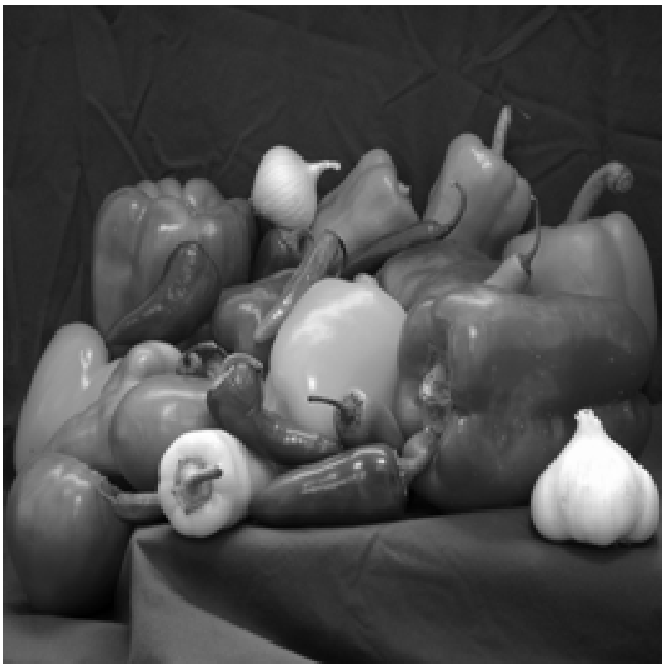}
            \caption{}\label{fig:pep_groundtruth}
        \end{subfigure}
        \begin{subfigure}{0.25\textwidth}
            \centering
            \includegraphics[width=\textwidth]{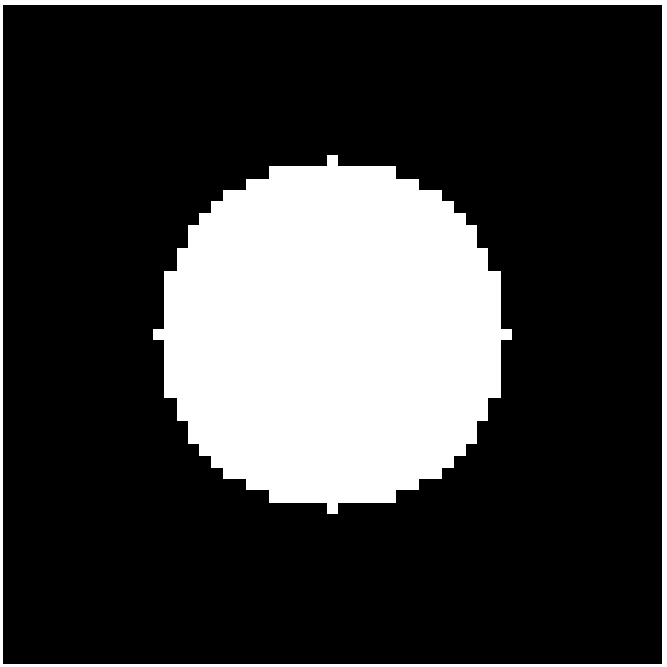}
            \caption{}\label{fig:psf_DefocusMedium}
        \end{subfigure}
        \begin{subfigure}{0.25\textwidth}
            \centering
            \includegraphics[width=\textwidth]{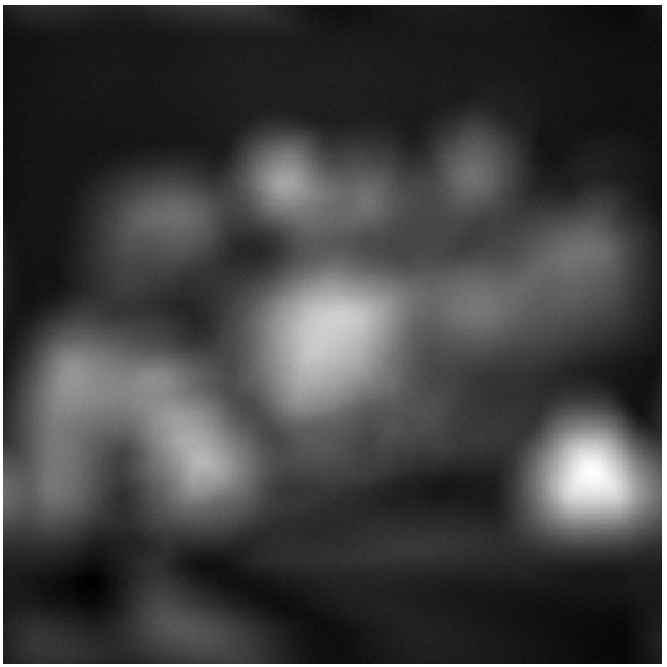}
            \caption{}\label{fig:b_DefocusMedium}
        \end{subfigure}
    }
    \caption{Example 3: (\subref{fig:pep_groundtruth}) true image,  (\subref{fig:psf_DefocusMedium}) defocus PSF of size $31\times 31$,  (\subref{fig:b_DefocusMedium}) observed image.}
    \label{fig:pep_DefocusMedium}
\end{figure}

The regularization parameter is chosen as $\lambda = 1 \cdot 10^{-4}$ for NPD and MNPD, and as $\lambda = 1 \cdot 10^{-3}$ for PNPD, MPNPD, and MPNPD\_auto.
In this experiment, the MPNPD\_auto algorithm performs $8$ coarse iterations for $19$ V-cycles, followed by $4$ coarse iterations for $38$ V-cycles.

\begin{figure}
     \centering
     \begin{overpic}[width=.6\linewidth]{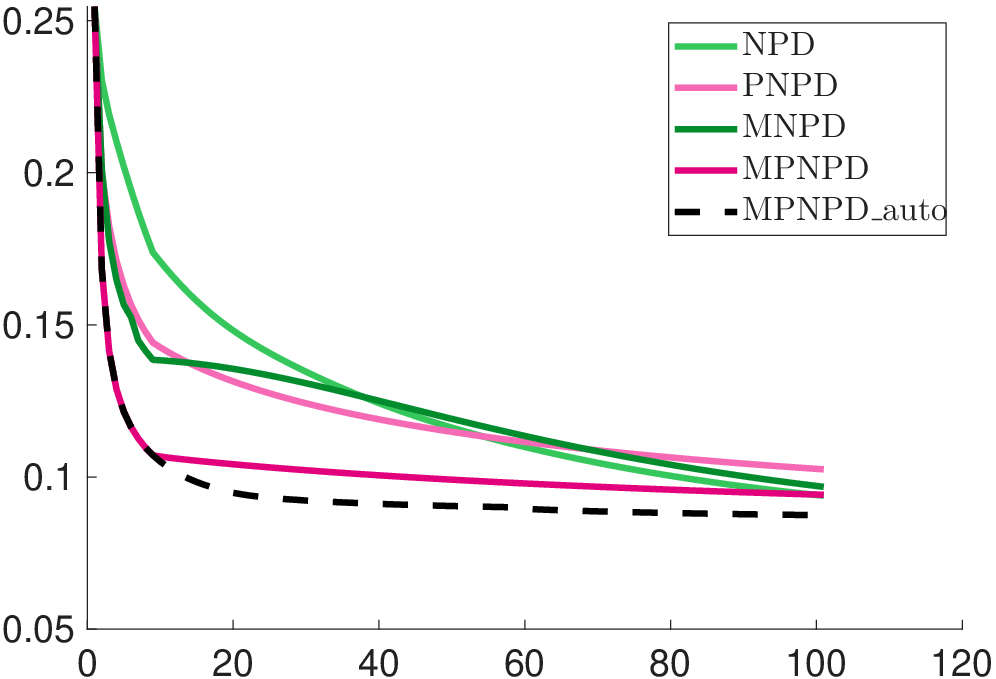}
        \put(32,-6){Number of iterations} 
        \put(-7,28){\rotatebox{90}{RRE}} 
    \end{overpic}
    \vspace{8mm}
     \caption{Example 3: comparison of the RRE varying the number of iterations for NPD, PNPD, and their multilevel versions.}
     \label{fig:err_pep_DefocusMedium}
\end{figure}

Figure~\ref{fig:err_pep_DefocusMedium} shows the comparison in terms of RRE for the five considered methods. The plot confirms that, also in the inexact case, the initial coarse-grid iterations performed by the multilevel methods significantly improve the overall performance. As for the MPNPD\_auto algorithm, increasing the number of V-cycles leads to a further reduction of the RRE; however, this improvement becomes only marginal after approximately the 20th iteration.

\begin{figure}
    \centering
	\begin{tabular}{>{\centering\arraybackslash}m{0.01\textwidth}
    >{\centering\arraybackslash}m{0.16\textwidth}
    >{\centering\arraybackslash}m{0.16\textwidth}
    >{\centering\arraybackslash}m{0.16\textwidth}
    >{\centering\arraybackslash}m{0.16\textwidth}
    >{\centering\arraybackslash}m{0.16\textwidth}}
    & \textcolor{npd}{NPD} & \textcolor{pnpd}{PNPD} & \textcolor{mnpd}{MNPD} & \textcolor{mpnpd}{MPNPD} & \textcolor{black}{MPNPD\_auto}\\
    \begin{tikzpicture}
    \node at (0,0) {};
    \node at (0,0.050\textwidth) {\rotatebox{90}{\textbf{Iter: 4}}};
    \end{tikzpicture} & \includegraphics[width=\linewidth]{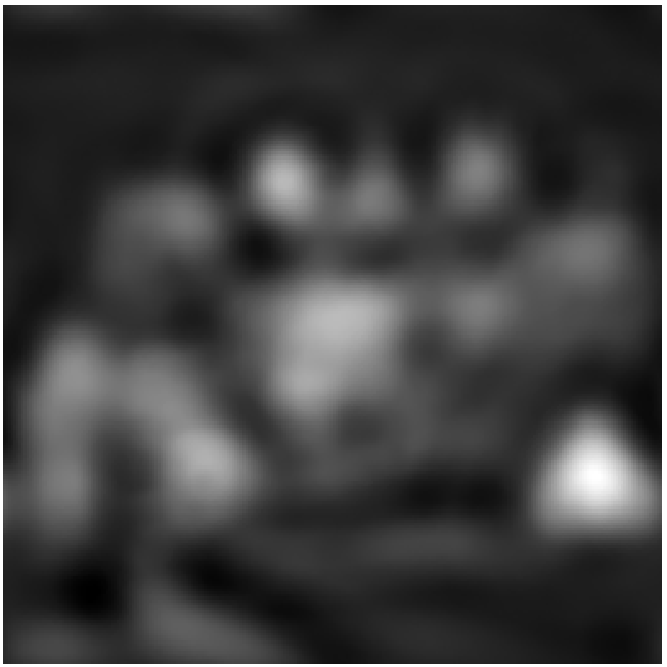}&
    \includegraphics[width=\linewidth]{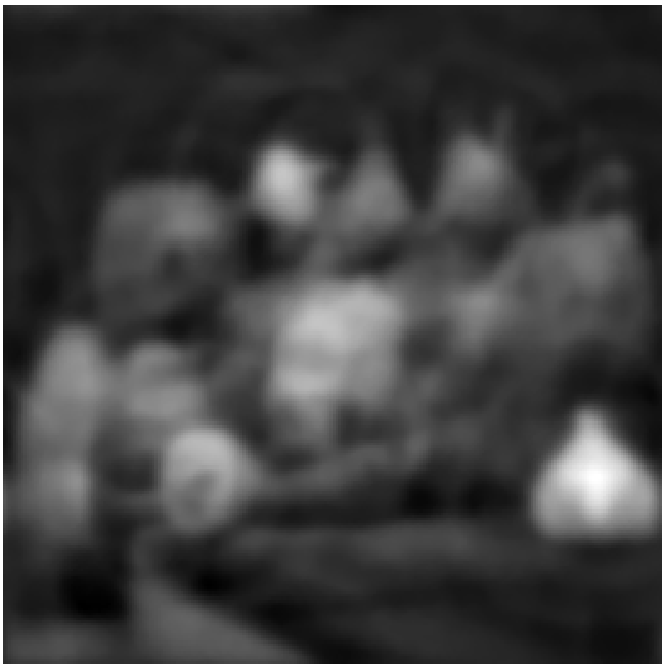} &
    \includegraphics[width=\linewidth]{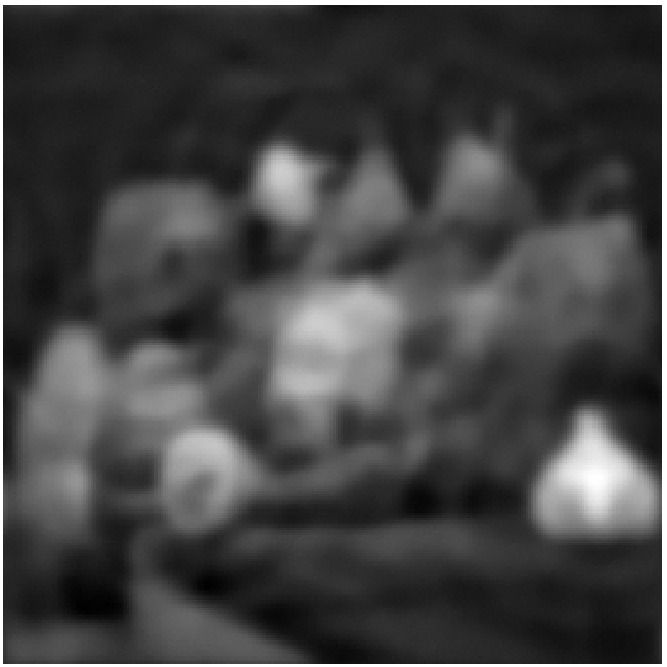}&
    \includegraphics[width=\linewidth]{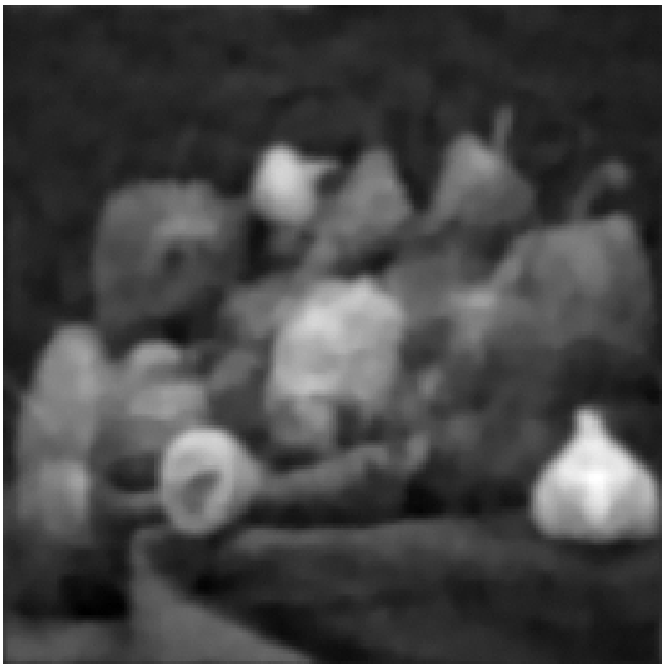} &
    \includegraphics[width=\linewidth]{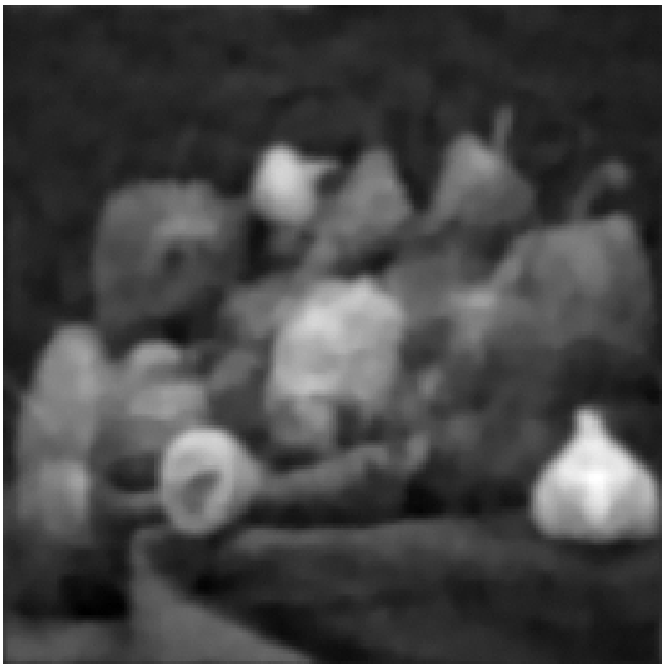}\\
    &\vspace*{-0.2cm}\text{\small 22.5627}&\vspace*{-0.2cm}\text{\small 24.4368}&\vspace*{-0.2cm}\text{\small 24.7746}&\vspace*{-0.2cm}\text{\small 26.9806}& \text{\small 26.9806}\\[0.2cm]
    \begin{tikzpicture}
    \node at (0,0) {};
    \node at (0,0.050\textwidth) {\rotatebox{90}{\textbf{Iter: 15}}};
    \end{tikzpicture}&\includegraphics[width=\linewidth]{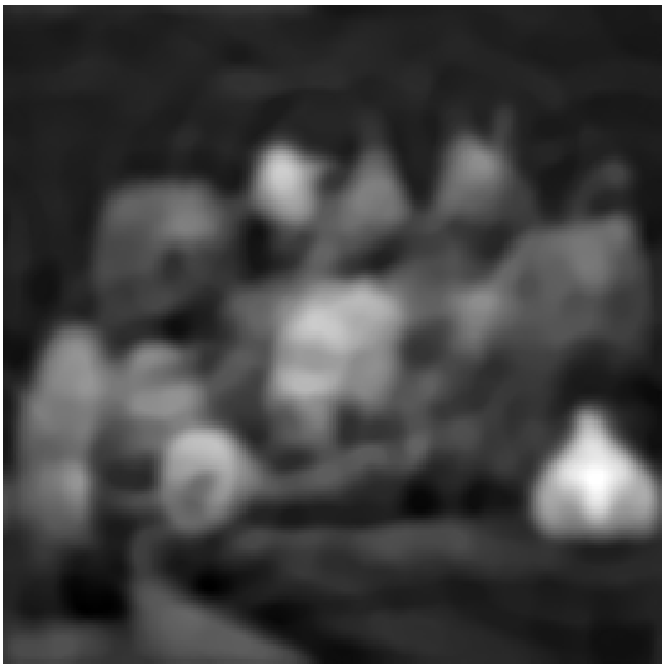}&
    \includegraphics[width=\linewidth]{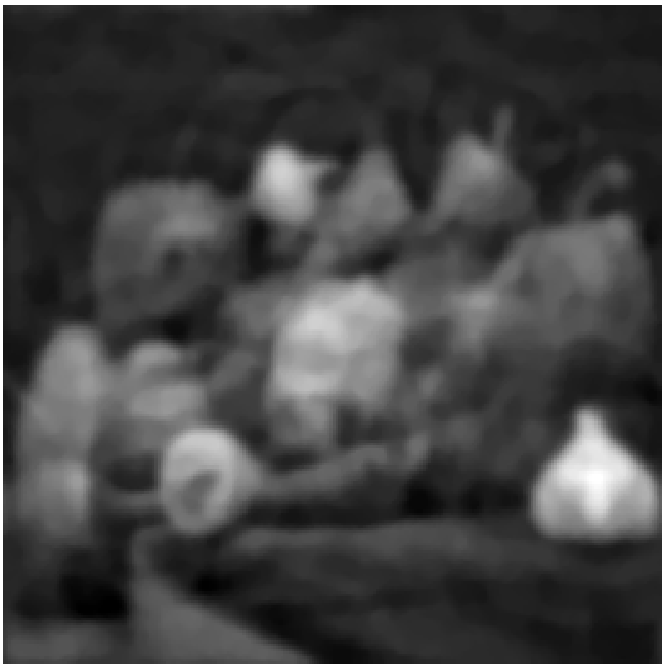} &
    \includegraphics[width=\linewidth]{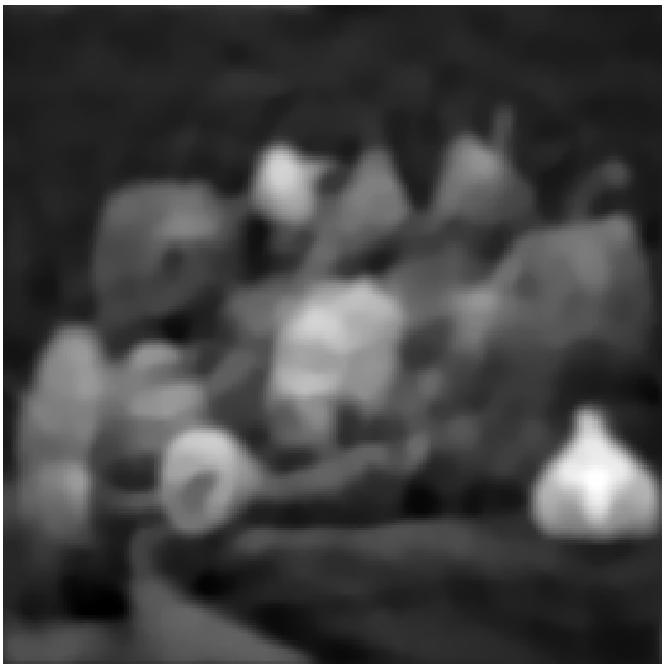}&
    \includegraphics[width=\linewidth]{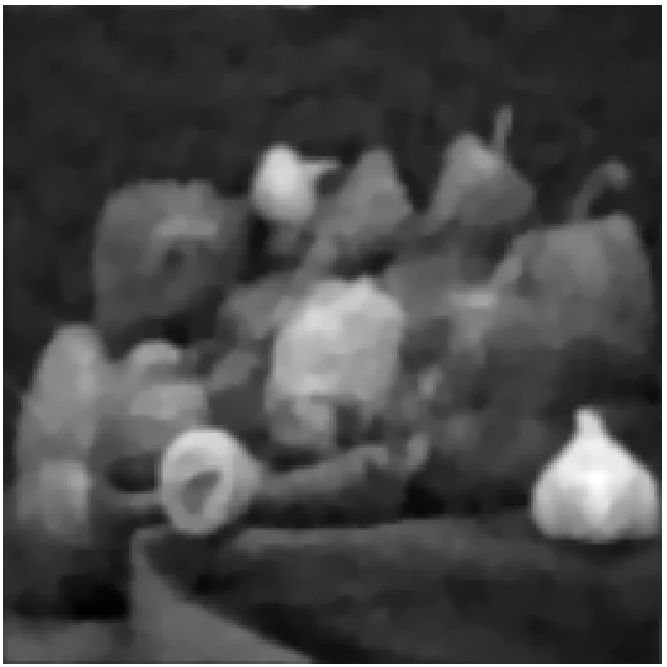}&
    \includegraphics[width=\linewidth]{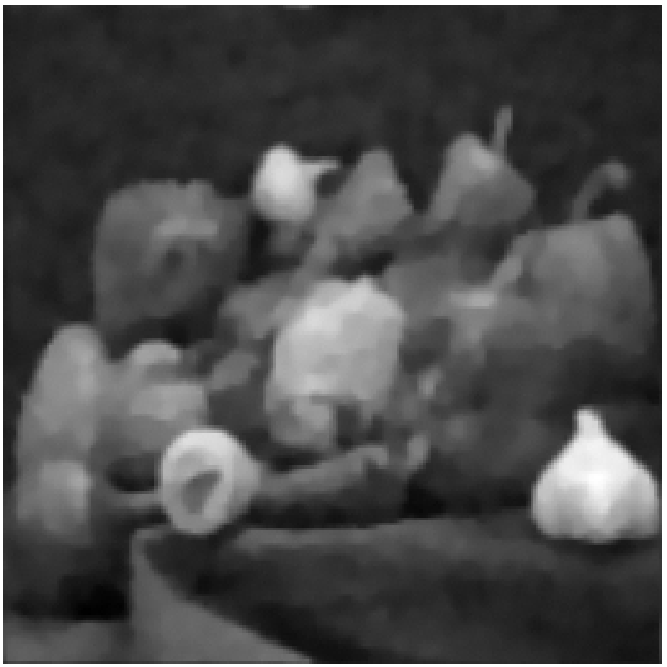}\\
    &\vspace*{-0.2cm}\text{\small 24.8281}&\vspace*{-0.2cm}\text{\small 26.0527}&\vspace*{-0.2cm}\text{\small 25.9426}&\vspace*{-0.2cm}\text{\small 28.2409}& \text{\small 28.8953}\\[0.2cm]
   \begin{tikzpicture}
    \node at (0,0) {};
    \node at (0,0.050\textwidth) {\rotatebox{90}{\textbf{Iter: 60}}};
    \end{tikzpicture}&\includegraphics[width=\linewidth]{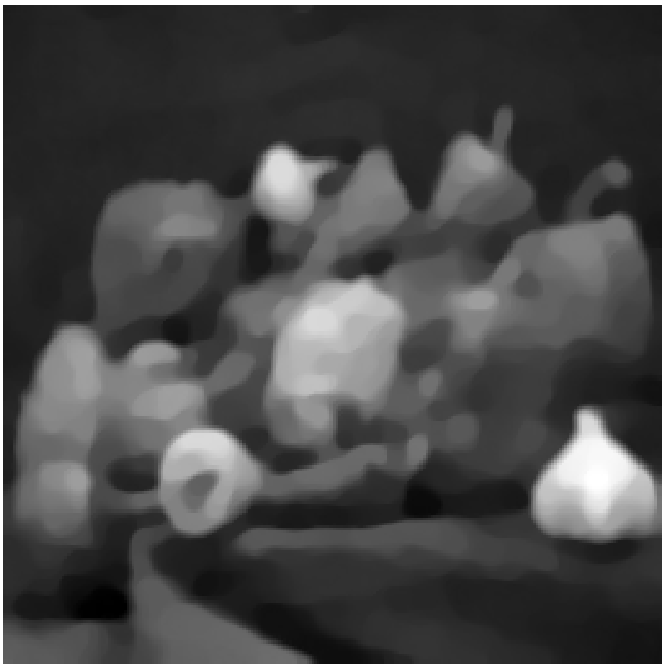}&
    \includegraphics[width=\linewidth]{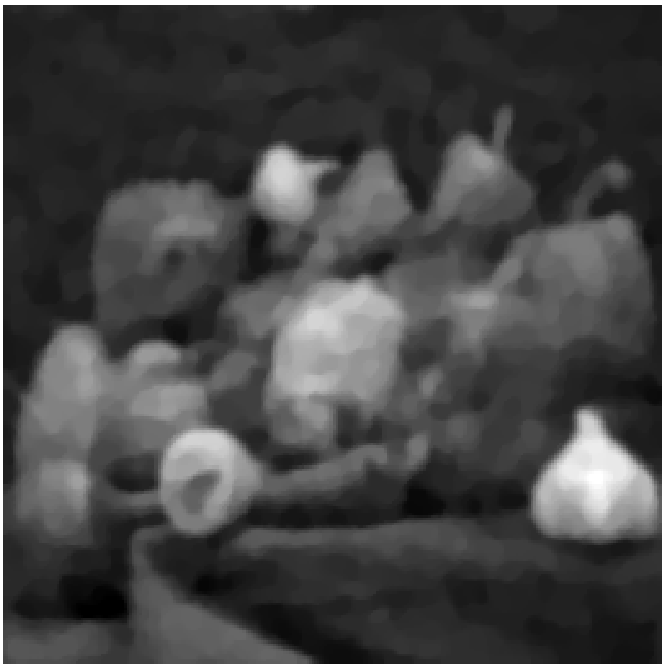} &
    \includegraphics[width=\linewidth]{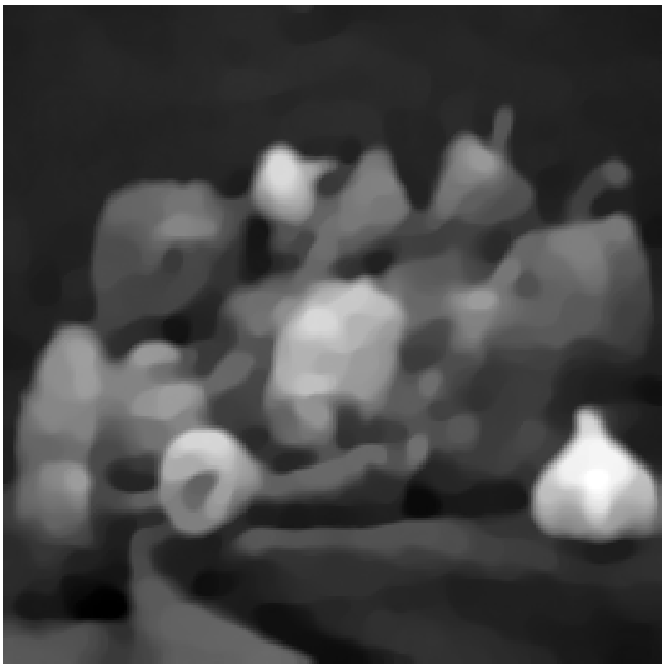}&
    \includegraphics[width=\linewidth]{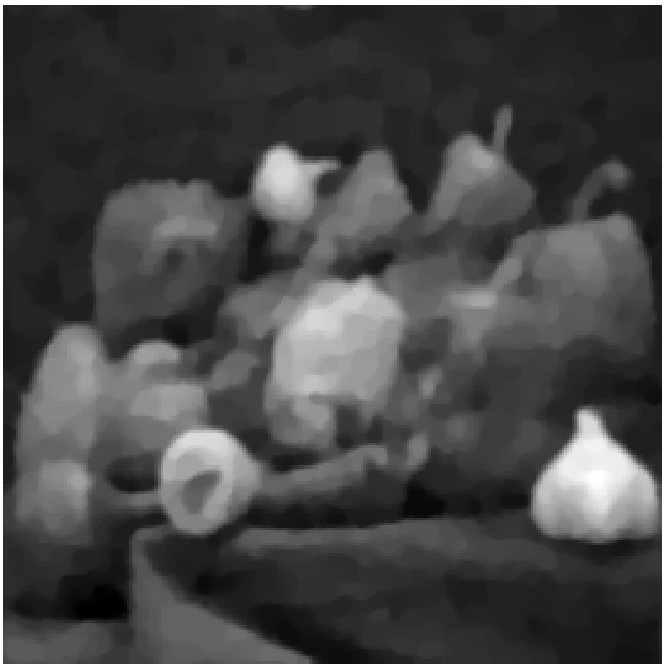}&
    \includegraphics[width=\linewidth]{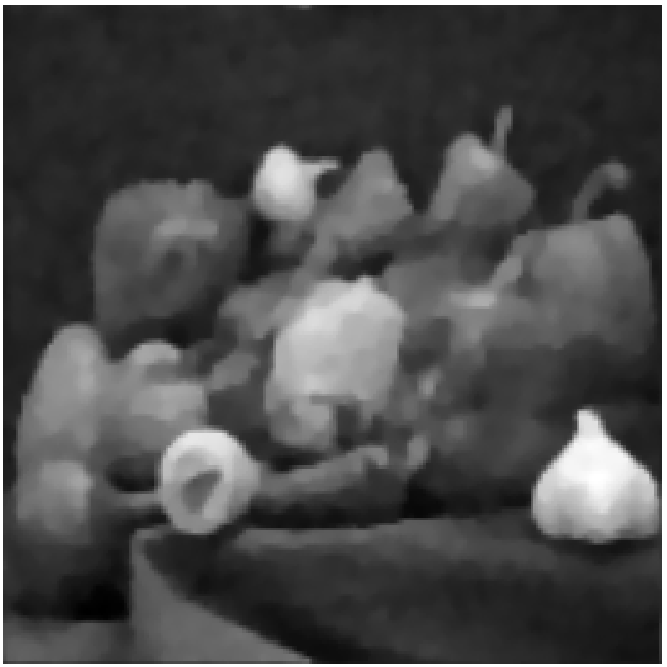}\\
    &\text{\small 27.9005}&\text{\small 27.7485}&\text{\small 27.6144}&\text{\small 28.8691}&\text{\small 29.6418} \\
    \end{tabular}
    \caption{Example 3: reconstructions obtained at different iterations with different methods.} 
    \label{fig:DefocusMedium_iterations}
\end{figure}
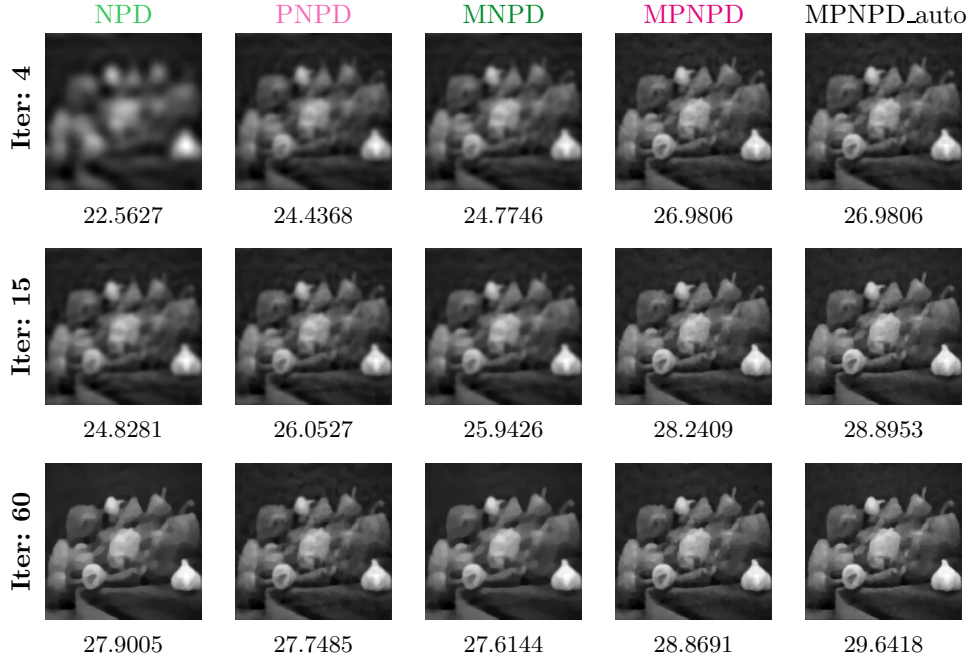

The comparison among the reconstructions obtained with NPD, PNPD, and their multilevel counterparts is reported in Figure~\ref{fig:DefocusMedium_iterations}, where the results after $4$, $15$, and $60$ iterations are shown. As in the previous examples, the multilevel preconditioned methods achieve the highest PSNR values among all the considered algorithms. In particular, after only $4$ iterations, MPNPD already provides a reconstruction of higher quality than the one obtained by standard PNPD after $15$ iterations.

\subsubsection{Example 4}
In this example, we consider a shake PSF of size $42\times 42$. Figure~\ref{fig:pep_ShakeSevere} displays the ground-truth image, the corresponding PSF, and the observed image.
\begin{figure}[!h]
    \centering
    \makebox[\textwidth][c]{
        \begin{subfigure}{0.25\textwidth}
            \centering
            \includegraphics[width=\textwidth]{Images/Peppers.eps}
            \caption{}
        \end{subfigure}
        \begin{subfigure}{0.25\textwidth}
            \centering
            \includegraphics[width=\textwidth]{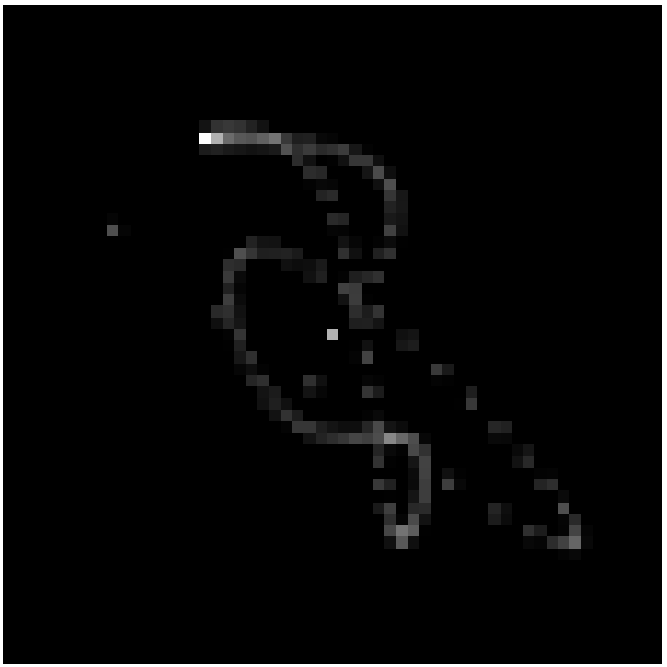}
            \caption{}\label{fig:psf_ShakeSevere}
        \end{subfigure}
        \begin{subfigure}{0.25\textwidth}
            \centering
            \includegraphics[width=\textwidth]{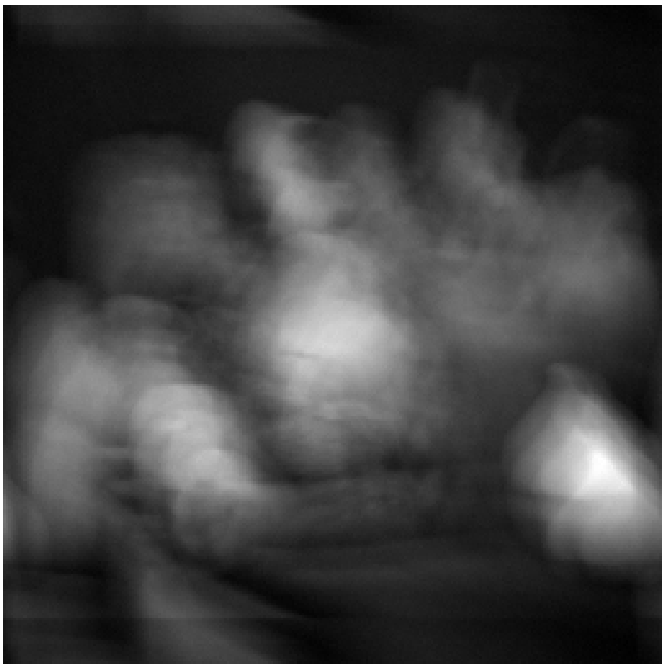}
            \caption{}\label{fig:b_ShakeSevere}
        \end{subfigure}
    }
    \caption{Example 4: (\subref{fig:pep_groundtruth}) true image, (\subref{fig:psf_ShakeSevere}) shake PSF of size $42\times 42$,  (\subref{fig:b_ShakeSevere}) observed image.}
    \label{fig:pep_ShakeSevere}
\end{figure}

The regularization parameter is chosen as $\lambda = 4 \cdot 10^{-4}$ for NPD and MNPD, and as $\lambda = 3 \cdot 10^{-3}$ for PNPD, MPNPD, and MPNPD\_auto.
In this example, the MPNPD\_auto algorithm performs $8$ coarse-level iterations for $7$ V-cycles, followed by $4$ additional coarse iterations for $2$ more V-cycles.

\begin{figure}
     \centering
    \begin{overpic}[width=.6\linewidth]{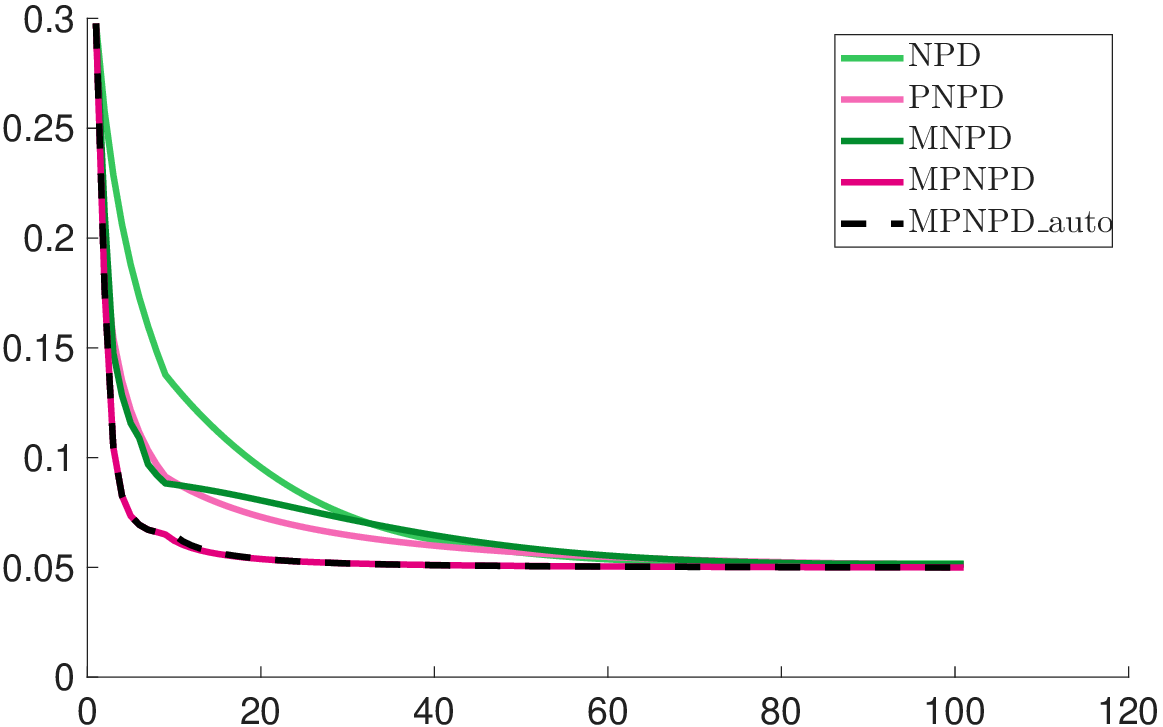}
        \put(32,-6){Number of iterations} 
        \put(-7,28){\rotatebox{90}{RRE}} 
    \end{overpic}
    \vspace{8mm}
     \caption{Example 4: comparison of the RRE varying the number of iterations for NPD, PNPD, and their multilevel versions.}
     \label{fig:err_pep_ShakeSevere}
\end{figure}

Figure~\ref{fig:err_pep_ShakeSevere} reports the comparison in terms of RRE for the considered methods. As in the previous examples, the plot confirms that the initial coarse-grid iterations performed by the multilevel methods provide a substantial acceleration of convergence. The same conclusions can be derived by looking at the restored images in
Figure~\ref{fig:ShakeSevere_iterations}.

\begin{figure}
    \centering
	\begin{tabular}{>{\centering\arraybackslash}m{0.01\textwidth}
    >{\centering\arraybackslash}m{0.16\textwidth}
    >{\centering\arraybackslash}m{0.16\textwidth}
    >{\centering\arraybackslash}m{0.16\textwidth}
    >{\centering\arraybackslash}m{0.16\textwidth}
    >{\centering\arraybackslash}m{0.16\textwidth}}
    & \textcolor{npd}{NPD} & \textcolor{pnpd}{PNPD} & \textcolor{mnpd}{MNPD} & \textcolor{mpnpd}{MPNPD} & \textcolor{black}{MPNPD\_auto}\\
    \begin{tikzpicture}
    \node at (0,0) {};
    \node at (0,0.050\textwidth) {\rotatebox{90}{\textbf{Iter: 4}}};
    \end{tikzpicture} & \includegraphics[width=\linewidth]{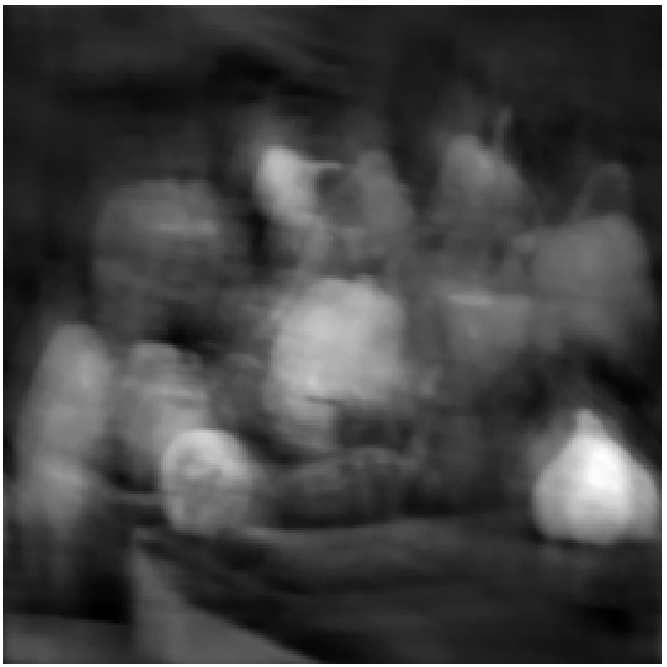}&
    \includegraphics[width=\linewidth]{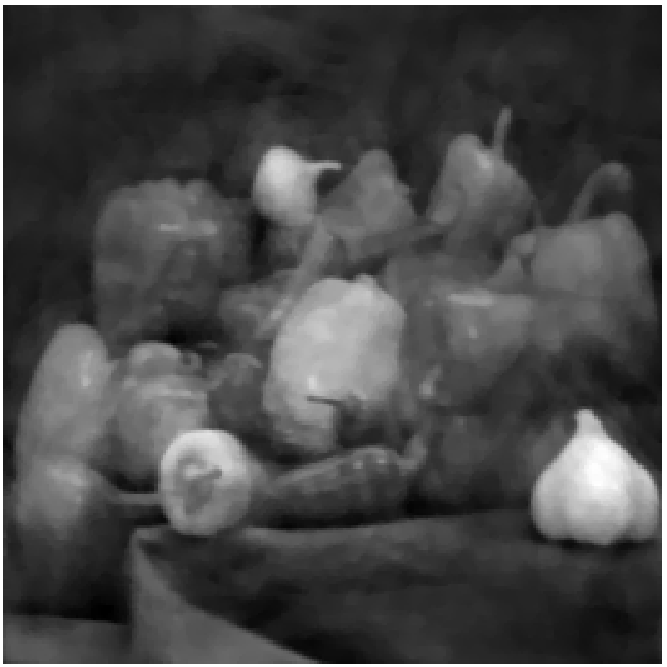} &
    \includegraphics[width=\linewidth]{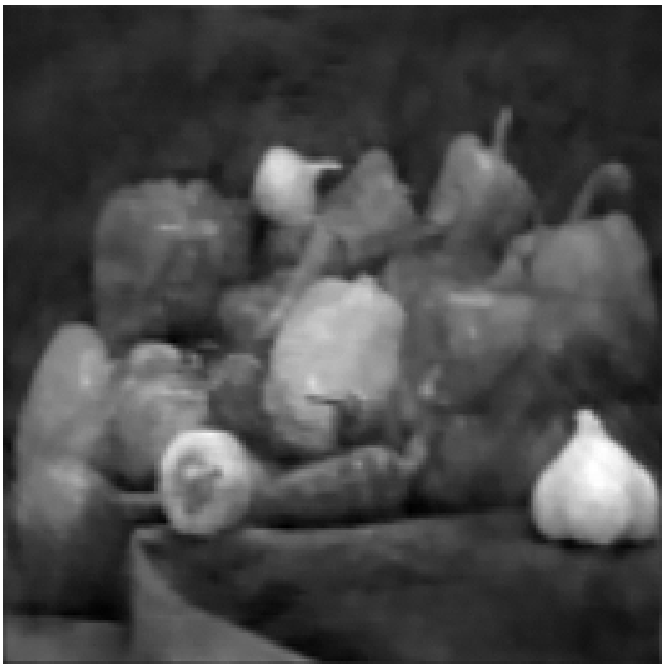}&
    \includegraphics[width=\linewidth]{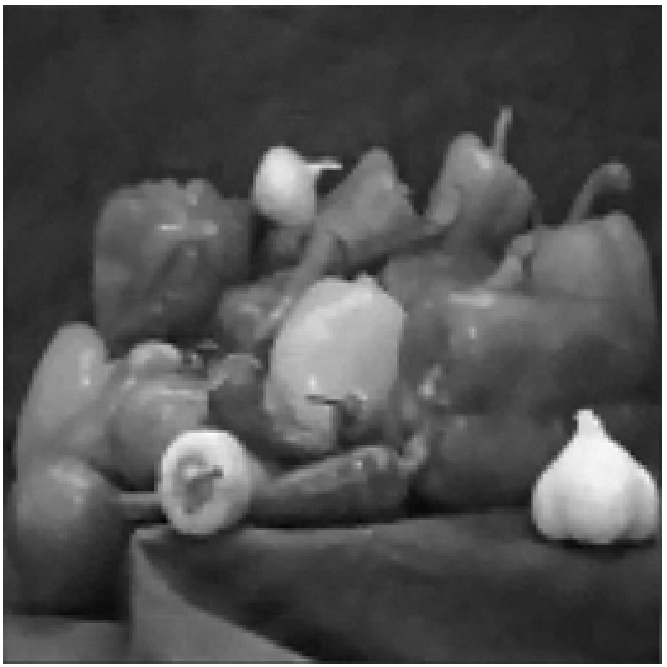} &
    \includegraphics[width=\linewidth]{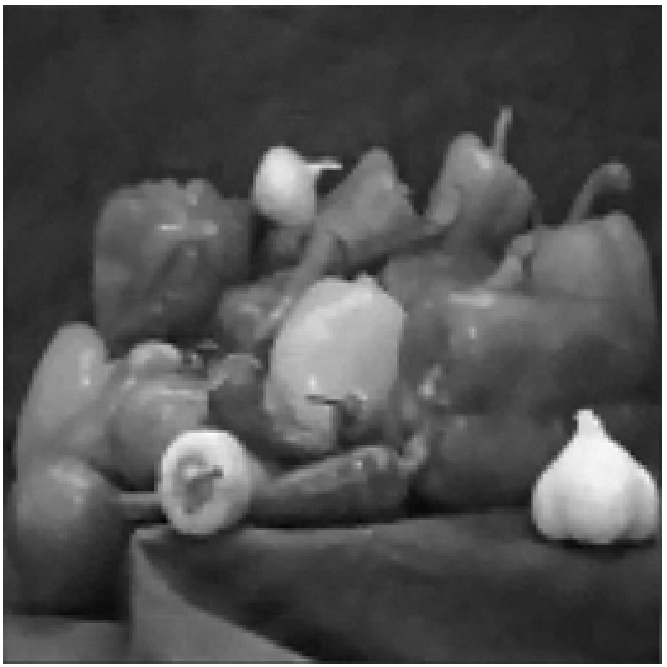}\\
    &\vspace*{-0.2cm}\text{\small 23.1924}&\vspace*{-0.2cm}\text{\small 26.9759}&\vspace*{-0.2cm}\text{\small 27.4100}&\vspace*{-0.2cm}\text{\small 31.35599}& \text{\small 31.3559}\\[0.2cm]
    \begin{tikzpicture}
    \node at (0,0) {};
    \node at (0,0.050\textwidth) {\rotatebox{90}{\textbf{Iter: 15}}};
    \end{tikzpicture}&\includegraphics[width=\linewidth]{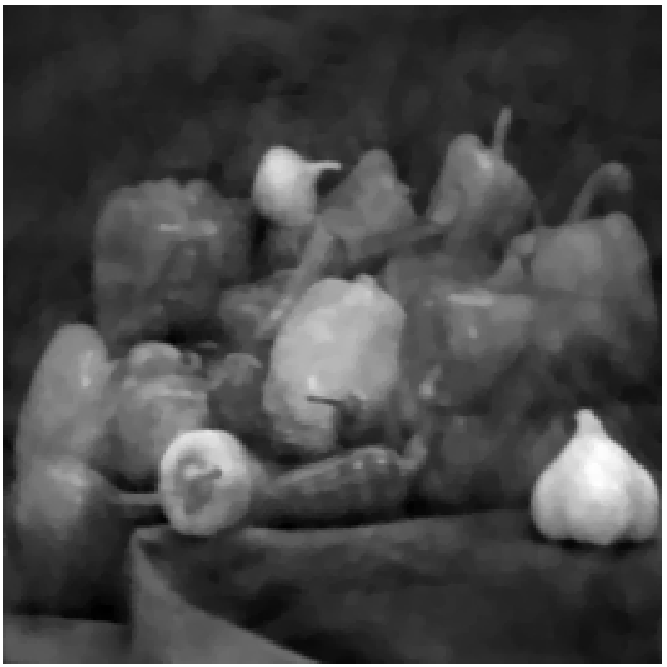}&
    \includegraphics[width=\linewidth]{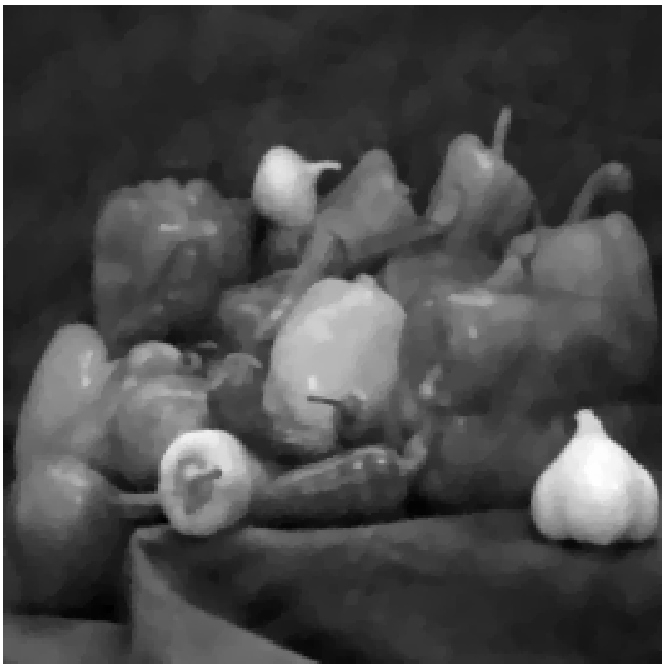} &
    \includegraphics[width=\linewidth]{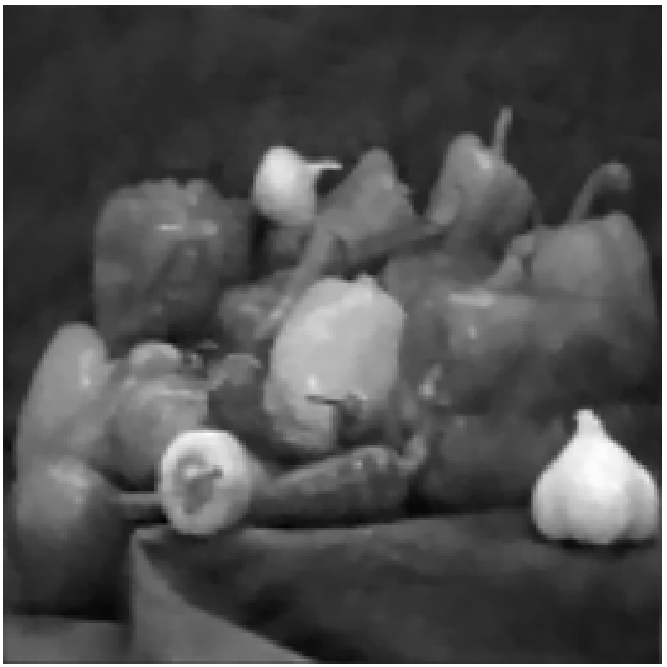}&
    \includegraphics[width=\linewidth]{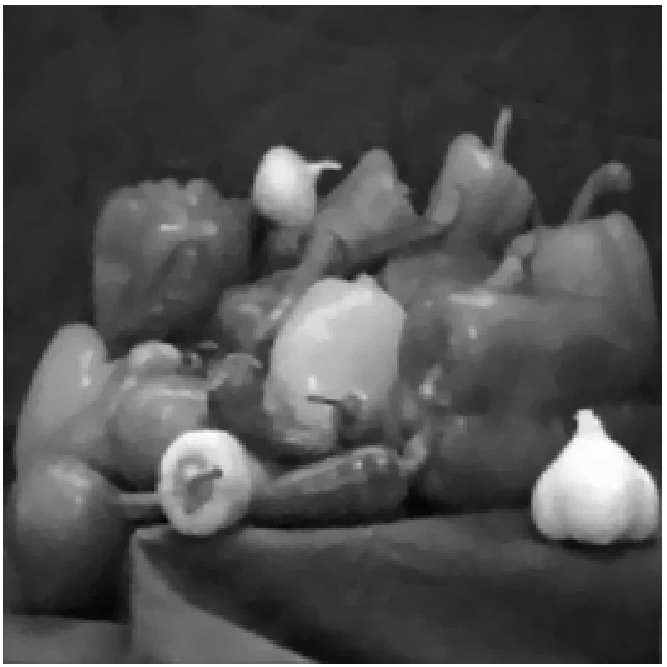}&
    \includegraphics[width=\linewidth]{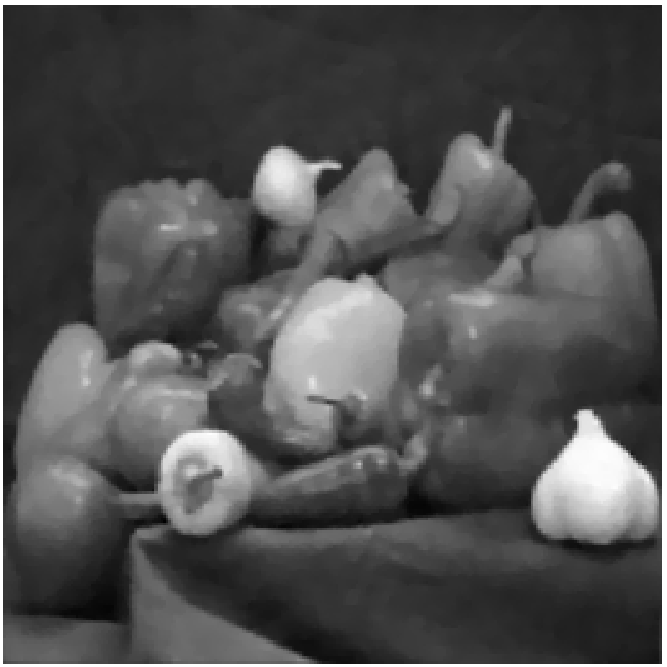}\\
    &\vspace*{-0.2cm}\text{\small 27.9786}&\vspace*{-0.2cm}\text{\small 30.8264}&\vspace*{-0.2cm}\text{\small 30.2138}&\vspace*{-0.2cm}\text{\small 33.7781}& \text{\small 33.7233}\\[0.2cm]
   \begin{tikzpicture}
    \node at (0,0) {};
    \node at (0,0.050\textwidth) {\rotatebox{90}{\textbf{Iter: 60}}};
    \end{tikzpicture}&\includegraphics[width=\linewidth]{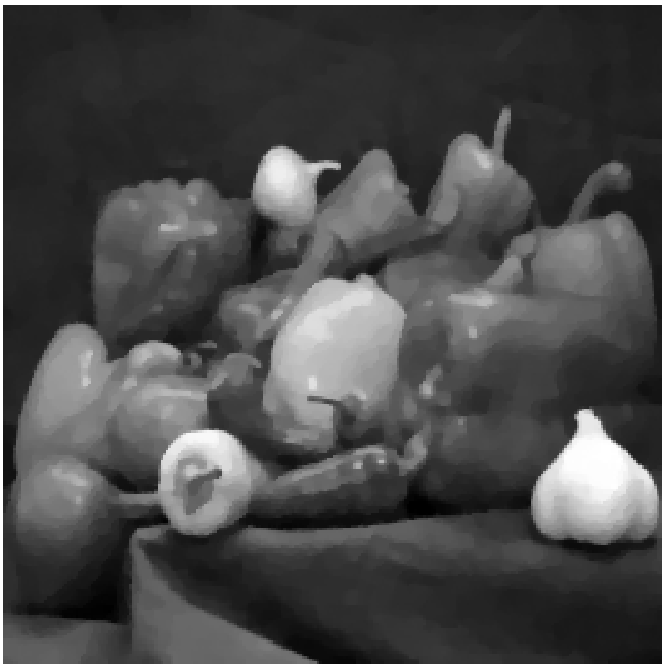}&
    \includegraphics[width=\linewidth]{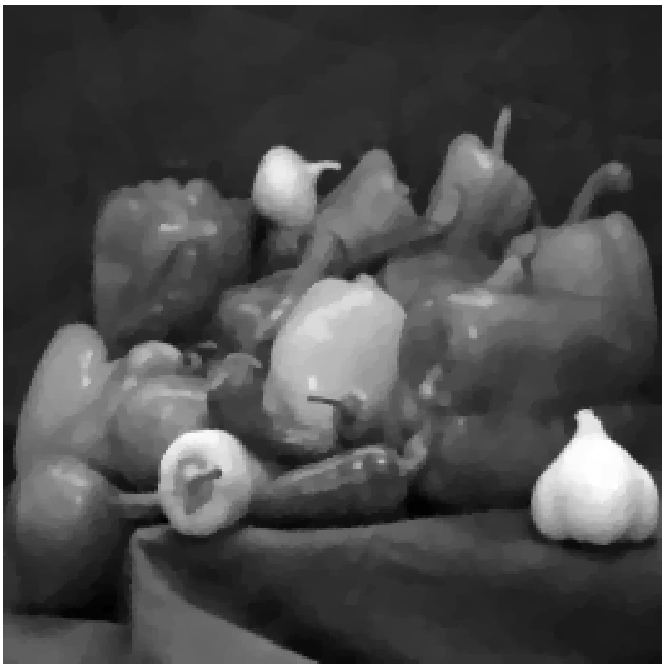} &
    \includegraphics[width=\linewidth]{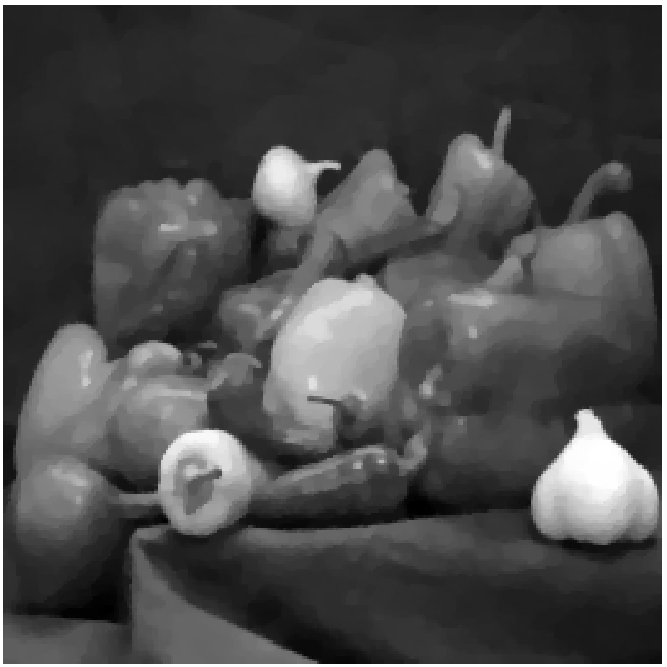}&
    \includegraphics[width=\linewidth]{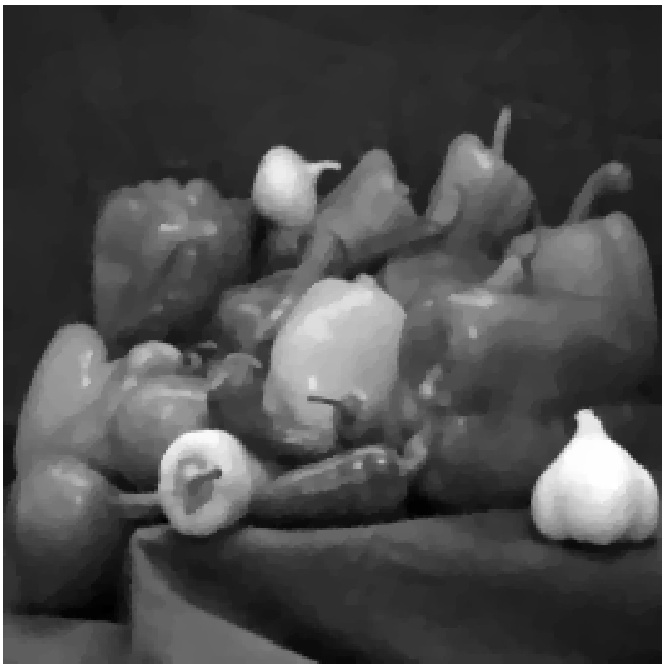}&
    \includegraphics[width=\linewidth]{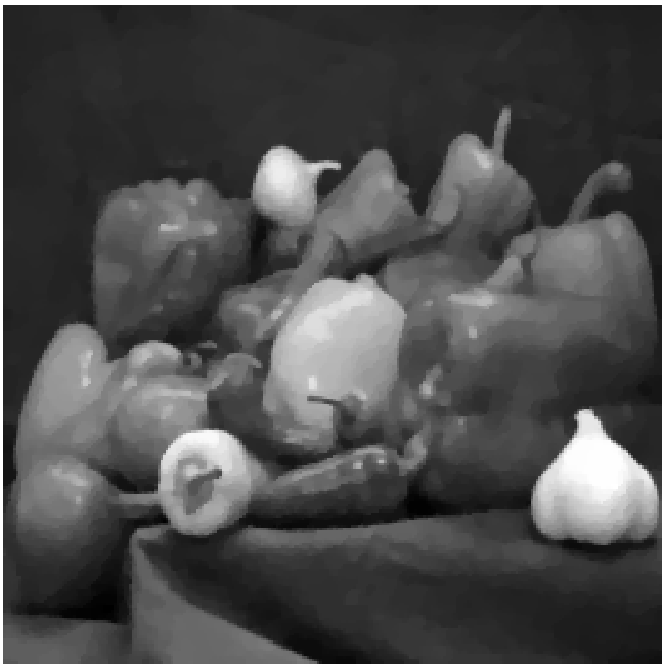}\\
    &\text{\small 34.1109}&\text{\small 33.9253}&\text{\small 33.8464}&\text{\small 34.6223}&\text{\small 34.6302} \\
    \end{tabular}
    \caption{Example 4: reconstructions obtained at different iterations with different methods.} 
    \label{fig:ShakeSevere_iterations}
\end{figure}

The main takeaway from these experiments is that the proposed methods are able to produce accurate approximations of the solution with only a small number of coarse-grid corrections and the aid of a suitable preconditioner, while remaining competitive with existing approaches.

\section{Conclusion} \label{sec:conc} 
Inspired by the IML FISTA method proposed in \cite{IML_FISTA}, we investigated the possibility of combining preconditioning strategies for proximal gradient methods with a multilevel framework.\
We proposed a convergent preconditioned multilevel framework for both exact and inexact proximal algorithms. Furthermore, we replaced condition \eqref{eq:decrease_smooth} with an Armijo condition in order to better control the corrections arising from the coarse level. Finally, we introduced an adaptive strategy for selecting the number of coarse iterations and V-cycles, based on the satisfaction of the Armijo condition. We tested our approach on several image deblurring problems involving observations corrupted by white Gaussian noise. 

The numerical experiments confirm the advantages of incorporating a preconditioner within a multilevel framework. Indeed, our methods perform very well across a wide range of point spread functions of different types and sizes. Specifically, the multilevel ITTA and PNPD algorithms provide high-quality reconstructions faster than their standard counterparts. This opens up promising opportunities for addressing large-scale problems. Furthermore, the automatic algorithm we propose allows us to reduce the number of parameters without compromising performance.

Future work will explore alternative choices for the preconditioning matrix, particularly for applications where our proposal may be computationally demanding to invert, such as in computed tomography. Moreover, investigating the link between coarse and fine-level optimization remains key to determining the most effective strategies for constructing lower-level models.

\section*{Declarations}

\begin{itemize}
\item Funding: The first three authors are members of the GNCS group of INdAM and are partially supported by the PRIN 2022 project “Inverse Problems in the Imaging Sciences (IPIS)” (2022ANC8HL) and by the INdAM-GNCS 2024 Project “Metodi strutturati per il signal processing avanzato” (CUP E53C25002010001).\\  The last author thanks the Swiss National Science Foundation for the support through the project “CardioTwin - Precision Cardiology based on Digital Twins” (Grant number 214817).
\item Conflict of interest: All authors certify that they have no affiliations with or involvement in any organization or entity with any financial interest or non-financial interest in the subject matter or materials discussed in this
manuscript. The authors have no Conflict of interest to declare that are relevant to the content of this article.
\item Ethics approval: The authors declare that research ethics approval was not required for this study.
\item Consent to participate: The authors declare that consent to participate was not required for this study.
\item Data availability: The test problems were generated using the IRtools toolbox \cite{ITtools}. The \texttt{peppers.png} test image is a standard dataset distributed with the MATLAB environment.
\item Code availability: The custom code generated during the current study is available from the corresponding author on reasonable request. 
\item Author contribution: The authors contributed equally to this work.
\end{itemize}

\bibliography{sn-bibliography}

@article{ISTA,
  title={An iterative thresholding algorithm for linear inverse problems with a sparsity constraint},
  author={Daubechies, Ingrid and Defrise, Michel and De Mol, Christine},
  journal={Communications on Pure and Applied Mathematics: A Journal Issued by the Courant Institute of Mathematical Sciences},
  volume={57},
  number={11},
  pages={1413--1457},
  year={2004},
  publisher={Wiley Online Library}
}

@article{FISTA,
  title={A fast iterative shrinkage-thresholding algorithm for linear inverse problems},
  author={Beck, Amir and Teboulle, Marc},
  journal={SIAM journal on imaging sciences},
  volume={2},
  number={1},
  pages={183--202},
  year={2009},
  publisher={SIAM}
}

@INPROCEEDINGS{MM_FISTA,
  author={Lauga, Guillaume and Riccietti, Elisa and Pustelnik, Nelly and Gonçalves, Paulo},
  booktitle={ICASSP 2023 - 2023 IEEE International Conference on Acoustics, Speech and Signal Processing (ICASSP)}, 
  title={Multilevel FISTA for Image Restoration}, 
  year={2023},
  volume={},
  number={},
  pages={1-5},
  doi={10.1109/ICASSP49357.2023.10094710}
}

@article{IML_FISTA,
author = {Lauga, Guillaume and Riccietti, Elisa and Pustelnik, Nelly and Gon\c{c}alves, Paulo},
title = {IML FISTA: A Multilevel Framework for Inexact and Inertial Forward-Backward. Application to Image Restoration},
journal = {SIAM Journal on Imaging Sciences},
volume = {17},
number = {3},
pages = {1347-1376},
year = {2024},
doi = {10.1137/23M1582345}
}

@article{NPD,
author = {Bonettini, S. and Prato, Marco and Rebegoldi, Simone},
year = {2022},
month = {08},
pages = {1-39},
title = {A nested primal–dual FISTA-like scheme for composite convex optimization problems},
volume = {84},
journal = {Computational Optimization and Applications},
doi = {10.1007/s10589-022-00410-x}
}

@article{NPDIT,
author = {Aleotti, Stefano and Bonettini, Silvia and Donatelli, Marco and Prato, Marco and Rebegoldi, Simone},
year = {2024},
month = {09},
pages = {357-395},
title = {A nested primal–dual iterated Tikhonov method for regularized convex optimization},
volume = {91},
journal = {Computational Optimization and Applications},
doi = {10.1007/s10589-024-00613-4}
}

@article{ITTA,
title = {Nonstationary iterated thresholding algorithms for image deblurring},
journal = {Inverse Problems and Imaging},
volume = {7},
number = {3},
pages = {717-736},
year = {2013},
issn = {1930-8337},
doi = {10.3934/ipi.2013.7.717},
url = {https://www.aimsciences.org/article/id/c4e14239-435b-4853-bda0-22909360edae},
author = {Jie Huang and Marco Donatelli and Raymond H. Chan}
}

@article{PNPD,
title = {A Preconditioned Version of a Nested Primal-Dual Algorithm for Image Deblurring},
author = {Aleotti, Stefano and Donatelli, Marco and Krause, Rolf and Scarlato, Giuseppe},
journal = {Journal of Scientific Computing},
volume = {103},
number = {3},
pages = {85},
year = {2025},
publisher = {Springer},
doi = {10.1007/s10915-025-02863-8},
}

@article{ITtools,
author = {Gazzola, Silvia and Hansen, Per Christian and Nagy, James},
year = {2019},
month = {07},
pages = {},
title = {IR Tools: A MATLAB Package of Iterative Regularization Methods and Large-Scale Test Problems},
volume = {81},
journal = {Numerical Algorithms},
doi = {10.1007/s11075-018-0570-7}
}

@book{rockafellar1997convex,
  address = {Princeton, N. J.},
  author = {Rockafellar, R. Tyrrell},
  callnumber = {UniM Maths 516.08 R59},
  description = {robotica-bib},
  notes = {A SRL reference.},
  publisher = {Princeton University Press},
  series = {Princeton Mathematical Series},
  title = {Convex analysis},
  year = 1970
}

@article{Moreau_fr,
     author = {Moreau, J.J.},
     title = {Proximit\'e et dualit\'e dans un espace hilbertien},
     journal = {Bulletin de la Soci\'et\'e Math\'ematique de France},
     pages = {273--299},
     year = {1965},
     publisher = {Soci\'et\'e math\'ematique de France},
     volume = {93},
     doi = {10.24033/bsmf.1625},
     mrnumber = {34 #1829},
     zbl = {0136.12101},
     language = {fr},
     url = {https://www.numdam.org/articles/10.24033/bsmf.1625/}
}

@book{engl1996regularization,
  title        = {Regularization of Inverse Problems},
  author       = {Engl, Heinz W. and Hanke, Martin and Neubauer, Andreas},
  series       = {Mathematics and Its Applications},
  volume       = {375},
  publisher    = {Kluwer Academic Publishers},
  address      = {Dordrecht, The Netherlands},
  year         = {1996},
  isbn         = {9780792341574},
  doi          = {10.1007/978-94-009-1740-8},
}

@article{Bach-2012,
	author    = {F. Bach and R. Jenatton and J. Mairal and G. Obozinski},
	title    	= {Structured Sparsity through Convex Optimization},
	year    	= {2012},
	journal   = {Stat. Sci.},
	pages    	= {450--468},
	volume    = {27},
	number		= {4}
}

@book{Bertero1998b,
	author    = {M. Bertero and P. Boccacci},
	title    	= {Introduction to inverse problems in imaging},
	year    	= {1998},
	publisher = {Institute of Physics Publishing},
	address		= {Bristol}
}

@book{HNO,
  title={Deblurring images: matrices, spectra, and filtering},
  author={Hansen, Per Christian and Nagy, James G and O'Leary, Dianne P},
  year={2006},
  publisher={SIAM},
  address={Philadelphia}
}

@book{Engl1996-yp,
  title     = {Regularization of Inverse Problems},
  author    = {Engl, Heinz W and Hanke, Martin and Neubauer, Gunther},
  publisher = {Springer},
  series    = {Mathematics and Its Applications},
  edition   = 1996,
  month     = jul,
  year      = 1996,
  address   = {Dordrecht, Netherlands},
  language  = {en}
}

@ARTICLE{Rudin-Osher-Fatemi-1992,
author={L.I. Rudin and S. Osher and E. Fatemi},
title={Nonlinear total variation based noise removal algorithms},
journal={J. Phys. D.},
year={1992},
volume={60},
number={1--4},
pages={259--268}
}

@article{Dossal2015stability,
author = {Aujol, J.-F and Dossal, Ch},
year = {2015},
month = {11},
pages = {2408-2433},
title = {Stability of Over-Relaxations for the Forward-Backward Algorithm, Application to FISTA},
volume = {25},
journal = {SIAM Journal on Optimization},
doi = {10.1137/140994964}
}

@article{Beck2012Smoothing,
author = {Beck, Amir and Teboulle, Marc},
title = {Smoothing and First Order Methods: A Unified Framework},
journal = {SIAM Journal on Optimization},
volume = {22},
number = {2},
pages = {557-580},
year = {2012},
doi = {10.1137/100818327},

}

@article{Polson-et-al-2015,
	author    = {N.G. Polson and J.G. Scott and B.T. Willard},
	title    	= {Proximal {A}lgorithms in {S}tatistics and
{M}achine {L}earning},
	year    	= {2015},
	journal   = {Stat. Sci.},
	pages    	= {559--581},
	volume    = {30},
	number		= {4}
}

@ARTICLE{Chambolle10,
  AUTHOR =       "A. Chambolle and T. Pock",
  TITLE =        "A First--order primal--dual algorithm for convex problems with applications to imaging",
   JOURNAL =    " J. Math. Imaging Vis.",
  YEAR =         "2011",
  volume =       "40",
  pages =        "120--145"
}

@ARTICLE{Malitsky-Pock-2018,
  AUTHOR =       {Y. Malitsky and T. Pock},
  TITLE =        {A {F}irst-{O}rder {P}rimal-{D}ual {A}lgorithm with {L}inesearch},
  JOURNAL =      {SIAM J. Optim.},
  YEAR =         {2018},
  volume =       {28},
  number=        {1},
  pages =        {411--432}
}

@ARTICLE{Chambolle-et-al-2024,
  AUTHOR =       {A. Chambolle and C. Delplancke and M. J. Ehrhardt and C.-B. Sch{\"o}nlieb and J. Tang},
  TITLE =        {Stochastic {P}rimal–{D}ual {H}ybrid {G}radient {A}lgorithm with {A}daptive {S}tep {S}izes},
  JOURNAL =      {J. Math. Imaging Vis.},
  YEAR =         {2024},
  volume =       {66},
  pages =        {294--313}
}

@article{Bonettini-et-al-2016a,
	author    = {S. Bonettini and F. Porta and V. Ruggiero},
	title    	= {A variable metric forward-backward method with extrapolation},
	year    	= {2016},
	journal   = {SIAM J. Sci. Comput.},
        volume = {38},
        pages = {A2558--A2584}
}

@article{Schmidt2011,
	author    = {M. Schmidt and N. Le Roux and F. Bach},
	title    	= {Convergence rates of inexact proximal-gradient methods for convex optimization},
	year    	= {2011},
	journal   = {arXiv:1109.2415v2}
}

@ARTICLE{Villa-etal-2013,
author={Villa, S. and Salzo, S. and Baldassarre, L. and Verri, A.},
title={Accelerated and inexact forward-backward algorithms},
journal={SIAM J. Optim.},
year={2013},
volume={23},
number={3},
pages={1607-1633}
}

@article{Bonettini2018a,
   author = {S. Bonettini and S. Rebegoldi and V. Ruggiero},
    title = {Inertial Variable Metric Techniques for the Inexact Forward--Backward Algorithm},
  journal = {SIAM J. Sci. Comput.},
volume = {40},
number = {5},
   pages   = {A3180--A3210},
     year = 2018
}

@article{Chen-et-al-2018,
	author    = {J. Chen and I. Loris},
	title    	= {On starting and stopping criteria for nested primal-dual iterations},
	year    	= {2019},
	journal   = {Numer. Algorithms},
        volume = {82},
        pages = {605--621}
}

@article{Chouzenoux-etal-2014,
	author    = {E. Chouzenoux and J.-C. Pesquet and A. Repetti},
	title    	= {Variable metric forward-backward algorithm for minimizing the
sum of a differentiable function and a convex function},
	year    	= {2014},
	journal   = {J. Optim. Theory Appl.},
	pages    	= {107--132},
	volume    = {162}
}

@ARTICLE{Frankel-etal-2015,
	author    = {P. Frankel and G. Garrigos and J. Peypouquet},
	title    	= {Splitting methods with variable metric for {{K}urdyka--{\L}ojasiewicz} functions and general convergence rates},
	year    	= {2015},
	journal   = {J. Optim. Theory Appl.},
	pages    	= {874--900},
	volume    = {165},
	number		= {3}
}

@article{Ghanbari-2018,
	author    = {H. Ghanbari and K. Scheinberg},
	title    	= {Proximal quasi-{N}ewton methods for regularized convex optimization with linear and accelerated sublinear convergence rates},
	year    	= {2018},
	journal   = {Comput. Optim. Appl.},
	pages    	= {597--627},
	volume    = {69}
}

@article{Lee2018,
	author    = {C. Lee and S. J. Wright},
	title    	= {Inexact Successive Quadratic Approximation for Regularized Optimization},
	year    	= {2019},
	journal   = {Comput. Optim. Appl.},
	pages    	= {641--674},
        volume  = {72}
}

@article{Beck-Teboulle-2009a,
      author={Beck, A  and Teboulle, M},
      title = {Fast gradient-based algorithms for constrained total variation image denoising and deblurring problems},
      journal = {IEEE Trans. Image Processing},
      year ={2009},
      volume = {18},
      number={11},
      pages ={2419--34}
}

@ARTICLE{Ochs-etal-2014,
  author    = {P. Ochs and Y. Chen and T. Brox and T. Pock},
  title     = {i{P}iano: Inertial Proximal Algorithm for Non-convex Optimization},
  journal   = {SIAM J. Imaging Sci.},
  year      = {2014},
  volume    = {7},
  number    = {2},
  pages     = {1388--1419}
}

@ARTICLE{Combettes-Wajs-2005,
  AUTHOR =       {P.L. Combettes and V.~R. Wajs},
  TITLE =        {Signal recovery by proximal forward-backward splitting},
  JOURNAL =      {Multiscale Model. Simul.},
  YEAR =         {2005},
  volume =       {4},
  number=        {4},
  pages =        {1168--1200}
}

@article{Daubechies-et-al-2004,
	author    = {I. Daubechies and M. Defrise and C. De Mol},
	title    	= {An iterative thresholding algorithm for linear inverse problems with a sparsity constraint},
	year    	= {2004},
	journal   = {Commun. Pure Appl. Math.},
	pages    	= {1413--1457},
	volume    = {57},
	number		= {11}
}

@article{cai2008framelet,
  title={A framelet-based image inpainting algorithm},
  author={Cai, Jian-Feng and Chan, Raymond H and Shen, Zuowei},
  journal={Applied and Computational Harmonic Analysis},
  volume={24},
  number={2},
  pages={131--149},
  year={2008},
  publisher={Elsevier}
}

@article{buccini2020multigrid,
  author  = {Alessandro Buccini and Marco Donatelli},
  title   = {A multigrid frame based method for image deblurring},
  journal = {Electron. Trans. Numer. Anal.},
  volume  = {53},
  year    = {2020},
  pages   = {283--312},
  doi     = {10.1553/etna_vol53s283},
}

@article{chan2010multilevel,
  title={A multilevel algorithm for simultaneously denoising and deblurring images},
  author={Chan, Raymond H and Chen, Ke},
  journal={SIAM Journal on Scientific Computing},
  volume={32},
  number={2},
  pages={1043--1063},
  year={2010},
  publisher={SIAM}
}

@article{donatelli2005multigrid,
  title={A multigrid for image deblurring with Tikhonov regularization},
  author={Donatelli, Marco},
  journal={Numerical linear algebra with applications},
  volume={12},
  number={8},
  pages={715--729},
  year={2005},
  publisher={Wiley Online Library}
}

@article{donatelli2006regularizing,
  title={On the regularizing power of multigrid-type algorithms},
  author={Donatelli, Marco and Serra-Capizzano, Stefano},
  journal={SIAM Journal on Scientific Computing},
  volume={27},
  number={6},
  pages={2053--2076},
  year={2006},
  publisher={SIAM}
}

@article{morigi2008cascadic,
  title={Cascadic multiresolution methods for image deblurring},
  author={Morigi, Serena and Reichel, Lothar and Sgallari, Fiorella and Shyshkov, Andriy},
  journal={SIAM Journal on Imaging Sciences},
  volume={1},
  number={1},
  pages={51--74},
  year={2008},
  publisher={SIAM}
}
\appendix
\section{Framelets}\label{app:framelets}
Let $W$ denote the tight frame associated with linear B-splines \cite{cai2008framelet}. The one-dimensional linear B-spline system consists of a low-pass filter $W_0$ and two high-pass filters $W_1$ and $W_2$. The corresponding masks are \[ w^{(0)} = \tfrac{1}{4}(1,\,2,\,1), \qquad w^{(1)} = \tfrac{\sqrt{2}}{4}(-1,\,0,\,1), \qquad w^{(2)} = \tfrac{1}{4}(-1,\,2,\,-1). \] Imposing reflexive boundary conditions, the resulting filter matrices are 
\[ W_0 = \frac{1}{4} \begin{bmatrix} 3 & 1 & 0 & \cdots & 0 \\ 1 & 2 & 1 & \cdots & 0 \\ \vdots & \vdots & \ddots & \vdots & \vdots \\ 0 & \cdots & 1 & 2 & 1 \\ 0 & \cdots & 0 & 1 & 3 \end{bmatrix}, \qquad W_1 = \frac{\sqrt{2}}{4} \begin{bmatrix} -1 & 1 & 0 & \cdots & 0 \\ -1 & 0 & 1 & \cdots & 0 \\ \vdots & \vdots & \ddots & \vdots & \vdots \\ 0 & \cdots & -1 & 0 & 1 \\ 0 & \cdots & 0 & -1 & 1 \end{bmatrix}, \] 
\[ W_2 = \frac{1}{4} \begin{bmatrix} 1 & -1 & 0 & \cdots & 0 \\ -1 & 2 & -1 & \cdots & 0 \\ \vdots & \vdots & \ddots & \vdots & \vdots \\ 0 & \cdots & -1 & 2 & -1 \\ 0 & \cdots & 0 & -1 & 1 \end{bmatrix}, \] 
and 
\begin{equation} \label{eq:framelets} 
W =\begin{bmatrix} W_0\\ W_1\\ W_2 \end{bmatrix}.
\end{equation} 
The operators $W_i$, $i=0,1,2$, are designed for one-dimensional signals.\\ 
The two-dimensional framelet filters are obtained via tensor products as 
\begin{equation*} W_{ij} = W_i \otimes W_j , \qquad i,j = 0,1,2, \end{equation*} 
and the corresponding 2D operator is 
\begin{equation} \label{eq:framelets_2D} 
W_{2D} = (W \otimes W) = \begin{bmatrix} W_{00}\\ W_{01}\\ \vdots\\ W_{22} \end{bmatrix}. \end{equation} 
The low-pass component is $W_{00}$, while all remaining matrices $W_{ij}$ contain at least one high-pass filter in at least one spatial direction.

\section{Computational cost}\label{app:computational_cost}
For simplicity, we set $d = n^2$, allowing us to represent a square image $X$ of size $n \times n$ at the fine level and an image of size $\frac{n}{2}\times \frac{n}{2}$ at the coarse level. Accordingly, the blurring matrix $A$ has size $n^2\times n^2$ at the fine level and $\frac{n^2}{4}\times \frac{n^2}{4}$ at the coarse level. The computational cost of one cycle of the multilevel scheme is given by
\begin{align*}
\mathcal{C}(CYCLE(n,m)) &= \mathcal{C}\left(SOLVER\left(\tfrac{n}{2}, \tfrac{n}{2}, m\right)\right) + \mathcal{P},
\end{align*}
where $\mathcal{C}\left(SOLVER\left(\frac{n}{2}, \frac{n}{2}, m\right)\right)$ denotes the number of flops for performing $m$ iterations of the solver at the coarse level, while $\mathcal{P}$ accounts for the cost of transferring information between fine and coarse levels.

\begin{itemize}
    \item \textbf{Transfer operators} \\
Each row of the restriction operator $I_h^H$ contains $9$ nonzero entries. Therefore, each row requires $8$ additions, $5$ multiplications, and $1$ division. For the prolongation operator $I_H^h$, one quarter of the fine-grid points coincide with coarse-grid points and are obtained by direct injection, requiring no arithmetic operations. For half of the points in the fine grid, which lie between two coarse grid nodes, we compute the weighted averages of $2$ neighbors, requiring $2$ operations per point. The remaining points lie in the center of a coarse grid cell, so they are computed as averages of $4$ neighbors, requiring $4$ operations.
Therefore, it holds
\[\mathcal{C}(I_h^Hx) = 14\frac{n^2}{4} = \frac{7}{2}n^2, \qquad \mathcal{C}(I_H^hx) = 0 + 4 \frac{n^2}{4} + 2 \frac{n^2}{2} = 2n^2, \]
obtaining
\[\mathcal{C}(I_H^hx) + \mathcal{C}(I_h^Hx) = 2n^2+\frac{7}{2}n^2 = \frac{11}{2}n^2, \]
and
\[\mathcal{P} = 2(\mathcal{C}(I_H^hx) + \mathcal{C}(I_h^Hx)) =11n^2,\]
since both the current iterate and the gradient of the objective function at the current iterate must be transferred.
\end{itemize}

Since the most expensive operations in all the considered methods are FFTs and multiplications by framelet operators, the cost analysis mainly focuses on these two components.

\begin{itemize}
    \item \textbf{Fast Fourier transform} \\
The cost of a one-dimensional FFT of a vector $x\in \mathbb{R}^n$ is
\[\mathcal{C}(FFT(n))= 5n\log_2 n.\]
Therefore, applying the two-dimensional FFT to a matrix $X\in \mathbb{R}^{n\times n}$ requires
\[\mathcal{C}(FFT2(n,n))= n\Big(5n\log_2n\Big)+n\Big(5n\log_2 n\Big)=10n^2\log_2n,\]
since one-dimensional FFT is applied to each row and column.\\
\item \textbf{Framelet multiplications} \\
The framelet operator $W\in \mathbb{R}^{3n\times n}$ is a sparse matrix defined in~\eqref{eq:framelets}, and the corresponding two-dimensional operator is defined in~\eqref{eq:framelets_2D}. By standard properties of the Kronecker product,
\begin{align*}
    W_{2D} x = (W \otimes W) x=\text{vec}(WXW^T), \qquad
    W_{2D}^Ty = (W^T \otimes W^T)y =\text{vec}(W^TYW),
\end{align*}
where $x=\text{vec}(X)$ and $y=\text{vec}(Y)$. Since the number of operations required to compute $WXW^T$ and $W^TYW$ is the same, it is sufficient to estimate the cost of $WXW^T$. Exploiting the structure of the filters $W_0$, $W_1$, and $W_2$, a direct count of the arithmetic operations yields
\begin{align*}
    \mathcal{C}(WX, n) = n(4n + 3n + 4n) = 11n^2, \qquad
    \mathcal{C}(YW^T, n) = 3n(4n + 3n + 4n) = 33n^2.
\end{align*}
In conclusion, we have
\[\mathcal{C}(W_{2D}x, n) =  \mathcal{C}(WX, n) + \mathcal{C}(YW^T, n)  = 44n^2.\]

\item \textbf{Solvers} \\
Under periodic boundary conditions, the matrix $A$ is block circulant with circulant blocks (BCCB) and is therefore diagonalized by the two-dimensional Fourier matrix, namely $A=F_2^H\Lambda F_2$, where $F_2=F_n\otimes F_n$ and $\Lambda$ is a diagonal matrix containing the eigenvalues of $A$. As a consequence, the eigenvalues of $A$ can be obtained by computing the FFT of its first column.

Exploiting this structure, the gradient $\nabla f = A^T(Ax-b)$ can be computed using two two-dimensional FFTs. The same cost applies to the computation of $P^{-1}A^T(Ax-b)$. Hence,
\begin{align*}
\mathcal{C}(FISTA(n, n,  m)) &= m(2\mathcal{C}(FFT2(n,n)) + k_Fn^2+ 2\mathcal{C}(W_{2D}x, n)) \\
&= m(20n^2\log_2n+k_Fn^2+88n^2)\\
\mathcal{C}(ISTA(n, n,  m)) &= m(2\mathcal{C}(FFT2(n,n)) + k_In^2+ 2\mathcal{C}(W_{2D}x, n)) \\
&= m(20n^2\log_2n+k_In^2+88n^2)\\
\mathcal{C}(NPD(n, n,  m)) &= m(2\mathcal{C}(FFT2(n,n)) + k_Nn^2+ 3\mathcal{C}(W_{2D}x, n)) \\
&= m(20n^2\log_2n+k_Nn^2+132n^2)\\
\mathcal{C}(PNPD(n, n,  m)) &= m(2\mathcal{C}(FFT2(n,n)) + k_Pn^2+ 3\mathcal{C}(W_{2D}x, n)) \\
&= m(20n^2\log_2n+k_Pn^2+132n^2)
\end{align*}
where $k_i$, for $i=F,I,N,P$, collects the additional element-wise operations required by each algorithm, and $m$ denotes the number of iterations. For $m$ coarse-level iterations, it hold
\begin{align*}
\mathcal{C}\left(SOLVER(\tfrac{n}{2}, \tfrac{n}{2}, m)\right) &= m\left(20\frac{n^2}{4}\log_2\frac{n}{2}+k_1\frac{n^2}{4}+k_2\frac{n^2}{4}\right)\\
&=\frac{m}{4}\left(20n^2\log_2\frac{n}{2}+k_1n^2+k_2n^2\right)\\
&=\frac{m}{4}\left(20n^2\log_2n-20n^2+k_1n^2+k_2n^2\right)\\
&=\frac{1}{4} \mathcal{C}\left(SOLVER(n, n,  m)\right)-5n^2m.
\end{align*}
\end{itemize}

Combining the previous estimates, the computational cost of one multilevel cycle is
\begin{align*}
\mathcal{C}(CYCLE(n,m)) &= \mathcal{C}\Big(SOLVER(\tfrac{n}{2}, \tfrac{n}{2}, m)\Big) + \mathcal{P},\\
&=\frac{1}{4} \mathcal{C}\left(SOLVER(n, n,  m)\right)-5n^2m+9n^2\\
& \approx  \frac{1}{4} \mathcal{C}(SOLVER(n, n,  m)).
\end{align*}
In particular, $\mathcal{C}(CYCLE(n,m))<\frac{1}{4} \mathcal{C}(SOLVER(n, n,  m))$ whenever $m\geq 3$.

\end{document}